\documentclass{article}
\usepackage[nottoc]{tocbibind}
\usepackage[titletoc,title]{appendix}
\usepackage[utf8]{inputenc}

\usepackage[pdftex]{graphicx}
\usepackage[pdftex,dvipsnames,usenames]{color}

\usepackage{enumitem}
\usepackage{mathtools}

\usepackage{cancel}

\usepackage[english]{babel}
\usepackage[hyperindex,breaklinks]{hyperref}

\hypersetup{
  colorlinks   = true, 
  urlcolor     = blue, 
  linkcolor    = blue, 
  citecolor   = red 
}

\usepackage{amsmath,amssymb,amsfonts,mathrsfs, amsthm}

\usepackage{graphicx}
\usepackage[margin=0.75in]{geometry}

\usepackage{soul}

\usepackage{pdfpages}

\usepackage{lipsum}
\usepackage{csquotes}
\usepackage{thmtools}

\usepackage[
  backend=biber,
  style=alphabetic,
  maxnames=5,       
  minnames=5,       
  maxbibnames=5,    
  minbibnames=5,     
  maxalphanames=5,  
  minalphanames=5
]{biblatex}
\usepackage{amsfonts,amsmath,amstext,amssymb,amsthm}
\usepackage{enumitem}
\usepackage[english]{babel}
\usepackage{fancybox}
\usepackage{xspace}
\usepackage{amscd}
\usepackage[all]{xy}
\usepackage[svgnames]{xcolor}
\usepackage{calc}

\usepackage{esint}

\usepackage{enumitem}

\usepackage{cases}
\usepackage{eqnarray}

\usepackage[compact]{titlesec}

\usepackage{accents}

\def\XXint#1#2#3{{\setbox0=\hbox{$#1{#2#3}{\int}$ }
\vcenter{\hbox{$#2#3$ }}\kern-.56\wd0}}

\newcommand{\eps}{\varepsilon}

\newcommand{\N}{\mathbb{N}}
\newcommand{\R}{\mathbb{R}}

\newcommand{\cH}{\mathcal{H}}

\newcommand{\Sp}{\mathbb{S}}

\newcommand{\ep}{\varepsilon}

\newcommand{\HH}{\operatorname{H}}

\usepackage{cleveref}

\newtheorem{theorem}{Theorem}[section]
\newtheorem*{remarkx}{Remark}
\newtheorem{lconjecture}{Conjecture} 

\newtheorem{ltheorem}[lconjecture]{Theorem} 
\newtheorem{proposition}[theorem]{Proposition}
\newtheorem{lemma}[theorem]{Lemma}
\newtheorem{corollary}[theorem]{Corollary}

\theoremstyle{definition}

\newtheorem{definition}[theorem]{Definition}

\newtheorem{remark}[theorem]{Remark}

\usepackage{url}

\makeatletter
\renewenvironment{abstract}{
  \par\smallskip
  \noindent\textsc{Abstract.}\ %
}{
  \par\medskip
}
\makeatother

\makeatletter
\newcommand{\addressa}[1]{\gdef\@addressa{#1}}
\newcommand{\emaila}[1]{\gdef\@emaila{\url{#1}}}

\newcommand{\@endstuff}{\par\vspace{\baselineskip}\noindent
\begin{tabular}{@{}l}\scshape\@addressa\\\textit{E-mail address:} \@emaila\end{tabular} 

}

\AtEndDocument{\@endstuff}
\makeatother

\begin{document}

\setlength{\abovedisplayskip}{6pt}          
\setlength{\belowdisplayskip}{6pt}          
\setlength{\abovedisplayshortskip}{6pt}     
\setlength{\belowdisplayshortskip}{6pt}     

\setlength{\parskip}{0pt plus 1pt}
\setlength{\parindent}{15pt}
\setlength{\abovedisplayskip}{5pt plus 1pt minus 2pt}
\setlength{\belowdisplayskip}{5pt plus 1pt minus 2pt}

\setlist[itemize]{itemsep=2pt, topsep=3pt, partopsep=1pt, parsep=1pt}
\setlist[enumerate]{noitemsep, topsep=2pt, partopsep=0pt, parsep=0pt}

\titlespacing{\section}{0pt}{10pt plus 2pt minus 2pt}{6pt plus 1pt minus 1pt}
\titlespacing{\subsection}{0pt}{8pt plus 2pt minus 2pt}{5pt plus 1pt minus 1pt}

\title{\huge Phase transitions with bounded index:\\
Parallels to De Giorgi's conjecture}
\author{Enric Florit-Simon}
\date{}
\maketitle

\addressa{Enric Florit-Simon \\ Department of Mathematics, ETH Z\"{u}rich \\ Rämistrasse 101, 8092 Zürich, Switzerland}
\emaila{enric.florit@math.ethz.ch}
\newcommand{\enric}[1]{{\color{green}{#1}}}

\abstract{A well-known conjecture of De Giorgi---motivated by analogy with the Bernstein problem for minimal surfaces---asserts the rigidity of monotone solutions to the Allen--Cahn equation in $\R^{d+1}$, with $d\leq 7$.

We establish close parallels to De Giorgi's conjecture for general solutions of bounded Morse index, far stronger than the minimal surface analogy would suggest: Namely, any finite index solution to the Allen--Cahn equation with bounded energy density in $\R^4$ is one-dimensional, and---conditionally on the classification of stable solutions---the same holds for all $4\leq n \leq 7$.

As a geometric application, phase transitions with bounded energy and index in closed four-manifolds have smooth transition layers which behave like minimal hypersurfaces.

Consequently, phase transitions exhibit a remarkably rigid behaviour in higher dimensions. This is in stark contrast with the 3D case, in which a wealth of nontrivial entire solutions with finite index (and energy density) is conversely known to exist, by work pioneered by Del Pino--Kowalczyk--Wei. The authors conjectured that any such solution must have parallel ends which are either planar or catenoidal, suggesting it as a parallel to De Giorgi's conjecture in this framework. We confirm this picture under the bounded energy density assumption.
}

\tableofcontents

\section{Introduction}
\subsection{Minimal surfaces and phase transitions}
Minimal surfaces are the critical points of the area functional, and one of the most fundamental models of geometric optimisation.
Accordingly, a central problem in variational analysis and geometry concerns their existence, regularity and behaviour.

A closely related model is the Allen--Cahn (A--C) functional, introduced in the 1970s as a phase-field model for binary alloys \cite{AC72}. Let
\begin{equation}\label{first_time_energy}
    \mathcal{E}^\ep(u, \Omega) = \frac{1}{\sigma_{n-1}}\int_{\Omega} \left( \frac{\ep}{2}|\nabla u|^2 + \frac{1}{\ep} W(u) \right) dx,
\end{equation}
where $u: \Omega \to \mathbb{R}$, $\Omega \subset \mathbb{R}^n$ is open and $\sigma_{n-1}$ is defined in \eqref{eq:sig(n-1
)def}. We work with the explicit double-well potential $W(u) = \frac{1}{4}(1 - u^2)^2$ in this article.

Critical points $u_\eps$ of \( \mathcal{E}^\varepsilon \) solve the Allen--Cahn equation:
\begin{equation}\label{eq:epaceqintro}
    -\ep \Delta u_\ep + \frac{1}{\ep}W'(u_\ep) = 0.
\end{equation}
As a general philosophy, one expects the transition layers $\{|u_\ep|\leq 0.9\}$ to limit to minimal hypersurfaces as $\ep\to 0$, possibly in a generalised sense. This deep connection lies at the heart of a well known conjecture of De Giorgi \cite{DeG78}---a parallel to Bernstein's conjecture on minimal graphs in this setting---and its relevance in the development of the modern theories of the calculus of variations, phase transitions and $\Gamma$-convergence is hard to overstate:
\begin{lconjecture}[De Giorgi]\label{conj:originalDG}
    Let $3\leq n\leq 7$, and let $u:\R^{n+1}\to[-1,1]$ be a solution to the Allen--Cahn equation satisfying $\partial_{x_{n+1}} u>0$. Then, $u$ is one-dimensional.
\end{lconjecture}
The case $n=2$ was obtained in the works \cite{GG98,AC00}.\\

The goal of the present article is to study and clarify part of the conjectural picture---reviewed in \Cref{sec:wwredclas}---regarding the connection with minimal surfaces.

We record our main results in \Cref{sec:mainres}. While we confirm some long expected analogies, we also provide the strongest instances so far of \textit{divergence from minimal surface theory} in this setting, especially in higher dimensions. Based on this new evidence, we propose in \Cref{sec:mainres} a unifying ``finite index De Giorgi conjecture''---a significant strengthening of De Giorgi's original conjecture in dimensions $4$ to $7$ which we believe plausibly true.\\

In order to motivate our results, several prominent developments are worth highlighting.
\begin{itemize}
    \item The theory started from the perspective of \textit{energy minimisers} and \textit{$\Gamma$-convergence}, with work of Modica--Mortola as one of the first examples \cite{MM77}. More recently, and motivated by De Giorgi's influential conjecture \cite{DeG78}, by work of Caffarelli--C\'ordoba and Savin the regularity of energy minimisers has been shown to replicate that of area-minimising hypersurfaces in dimension up to $7$ \cite{CC95, CC06, Sav09}. The negative counterpart to the rigidity of minimisers in dimension $8$ was established by Pacard--Wei and Liu--Wang--Wei \cite{PW13,LWW17b} via the use of {\it gluing constructions}. Related techniques led to a negative result for De Giorgi's conjecture in dimension $9$, by work of Del Pino--Kowalczyk--Wei \cite{dPKW11}, as well to many other examples of global solutions \cite{dPKW13}.
    \item Another highlight of the early theory is Modica's discovery of a monotonicity formula for \textit{general critical points} \cite{Modica85,ModicaMonotonicity}. A deep connection with the theory of generalised minimal surfaces, i.e. stationary integral varifolds, has been subsequently established by Hutchinson, Tonegawa and Wickramasekera \cite{Hut86,HT00, Ton05, TonegawaWickramasekera2012}, exploiting fundamental breakthroughs in the regularity theory for the area \cite{Wic14}.
    \item From the \textit{geometric point of view}, there has been a recent explosion of research based around Yau's conjecture (on the existence of minimal hypersurfaces in closed manifolds), starting from the remarkably simple constructions by Gaspar and Guaraco of min-max solutions---which naturally satisfy \textit{bounds on their energies and Morse indices}---at the Allen--Cahn level \cite{Guaraco2018,GasparGuaraco2018,GGWeyl}. One of the most pressing issues is then a finer understanding of such solutions, with fundamental developments by Wang--Wei and Chodosh--Mantouldis in \cite{WW18, WW19, CM20}, and more recent work by Serra and the author in \cite{FSstable}.
\end{itemize}

The present article is very much motivated by these latter developments, and it uses the full power of the methods in \cite{WW18, WW19, CM20}. An additional common theme in our results (outlined in \Cref{sec:overproof}) is to go beyond the regime of \textit{scaling-invariant estimates}, in order to harness the interplay (and, in some cases, the \textit{mismatch}) between the natural scalings of minimal hypersurfaces and of solutions to the so-called Toda system (which governs the interactions between transition layers). This allows us to break through some fundamental obstructions exclusive to Allen--Cahn theory---some of which were already identified in \cite{FSstable}.

\subsection{The conjectural picture}\label{sec:wwredclas}
We review several influential conjectures and the progress towards them, emphasising throughout the link with minimal surface theory.

We recall that the \textit{Morse index} of a solution $u_\ep$ to \eqref{eq:epaceqintro} in $\Omega$ is
$$
\sup\left\{\dim E : E\subset C_c^1(\Omega)\ \mbox{is a linear subspace, and }
Q(\eta,\eta)<0\ \mbox{for all } \eta\in E\setminus\{0\}\right\},
$$
where
\begin{equation}\label{eq:2vardef}
    Q(\eta,\eta)=\frac{1}{\sigma_{n-1}}\int_{\Omega} \left(\ep|\nabla \eta|^2 + \frac{1}{\ep} W''(u_\ep)\eta^2 \right) dx\,.
\end{equation}

As we will explain now, stable solutions---i.e. those with zero index---are expected to be highly rigid, at least in low dimensions; on the other hand, solutions with infinite index are wildly flexible, as their zero level sets can have compact connected components of any topology \cite{EP16}. One expects the Morse index then to be an appropriate measure of complexity.

Unless otherwise stated, we will always consider a solution $u : \mathbb{R}^n \to [-1, 1]$ to the Allen--Cahn equation with $\varepsilon = 1$ (and denote $\mathcal E$ instead of $\mathcal E^1$ for its energy):
\begin{equation}\label{eq:aceqintro}
    -\Delta u + W'(u) = 0\,.
\end{equation}

A long-standing conjecture (see \cite[Problem 6]{Wang2021} and \cite{dPKW2009survey, dPKW12, ChanWei2018}), in analogy with by now classical results in the theory of minimal hypersurfaces (\cite{CSZ97,LW02}), is:
\begin{lconjecture}[Finite index implies finite ends]\label{conj:indimpends}
    Let $n\geq 3$, and let $u:\R^n\to[-1,1]$ be a finite Morse index solution to the Allen--Cahn equation. Then, $u$ has finitely many ends, in the sense that there is some $R_0>0$ such that $\{u=0\}$ has finitely many connected components in $\R^n\setminus \overline{B_{R_0}}$.
\end{lconjecture}
The case\footnote{In \cite{Wang2021} it is claimed that, \textit{assuming} the rigidity of stable solutions in $\R^3$, the techniques in \cite{WW18,GuiWangWei2020} can be used to prove \Cref{conj:indimpends} (i.e. the finiteness of ends) also for $n=3$. On the other hand, it is also explained in \cite{Wang2021} that these techniques cannot be applied in higher dimensions.} $n=2$ is a theorem instead by now, obtained in the remarkable article \cite{WW18}.
\begin{remarkx}[Yau's conjecture]
    The methods developed in the 2D case have already found further far-reaching applications, illustrating the interest behind the resolution of \Cref{conj:indimpends}. As a highlight, one of the main steps in the proof in \cite{WW18} is a precise description for stable solutions under appropriate sheeting assumptions. Chodosh and Mantoulidis strengthened and extended it to three dimensions in the breakthrough article \cite{CM20}, with which they first solved the multiplicity one and Morse index conjectures of Marques--Neves for Allen--Cahn approximations of minimal surfaces. This yields---in a generic setting---a strong form of Yau's pivotal conjecture on the existence of infinitely many minimal surfaces in closed three-dimensional manifolds \cite{Yau1982Seminar}. See the works \cite{MN16,IMN18,LMN18,Zhou2020,Song2023} and the references therein for an account of the developments which led to the full resolution of Yau's conjecture.
\end{remarkx}
In the minimal hypersurface case, much more is actually known. Denote points in $\R^{n-1}$ by $x'$, so that points in $\R^n$ are of the form $x=(x',x_n)$. Building on \cite{SSY75, SS81} and \cite{Schoen83, Anderson84}, the work of \cite{Tysk89} shows (see additionally \cite{FC85} for the case $n=3$):
\begin{ltheorem}\label{thm:TyskIndex}
    Let $3\leq n \leq 7$. Let $\Sigma\subset\R^n$ be an embedded minimal hypersurface with finite index, satisfying ${\rm Area}(\Sigma\cap B_R)\leq CR^{n-1}$ for some $C$ and all $R>0$. Then, there exist some $R_0>0$ and $N\in\N$, as well as constants $b_i,c_i$ for every $i=1,...,N$, such that: Up to a rotation, we can write
    \begin{equation}
        \Sigma \setminus \left(\overline{B_{R_0}'}\times[-R_0,R_0]\right)=\bigcup_{i=1}^N {\rm graph}\, f_i\,,
    \end{equation}
    where $f_1<...<f_N$.
    Moreover, if $n=3$ we have
    $$f_i=b_i+c_i\log|y|+O(|y|^{-1})\,,$$
    whereas if
    $4\leq n\leq 7$ we have
    $$f_i=b_i+c_i|y|^{3-n}+O(|y|^{2-n})\,.$$
\end{ltheorem}
\begin{remarkx}[Area density bounds]
    The area growth condition is often satisfied in common applications. Nevertheless, it was shown to actually be redundant in the ``classical'' case $n=3$ via \cite{DP79,FS80, Pog81} and \cite{Gul86,FC85,Osserman1964}. Much more recently, this has been established---together with the stable Bernstein conjecture, i.e. the classification of entire stable minimal hypersurfaces in $\R^n$ as hyperplanes---also for $n=4,5,6$ in \cite{CL24,CLMS25,Maz24} (see also \cite{CL23,CMR24}).

    It is worth remarking that the classification of stable entire solutions to Allen--Cahn---commonly known as the stable (or strong) De Giorgi conjecture, and which in particular would imply De Giorgi's original conjecture in one extra dimension by \cite{Jerison-Monneau,Sav09}---is instead wide open already for $n=3$. See however \cite{FS20, CFFS25} for a positive result for some closely related variants of the problem.
\end{remarkx}

It is natural to consider whether an appropriate analogue of \Cref{thm:TyskIndex} might hold in the Allen--Cahn case. For instance, we record conjecture (1) in \cite[p. 127]{dPKW13}, suggested by the authors as a parallel to De Giorgi's conjecture in this framework (see also \cite{dPKW2009survey, dPKW12}):
\begin{lconjecture}[Del Pino--Kowalczyk--Wei]\label{conj:dPKW}
    If $u$ [a solution to Allen--Cahn in $\R^3$] has finite Morse index and $\nabla u(x)\neq 0$ outside a bounded set, then, outside a large ball, each level set of $u$ must have a finite number of components, each of them asymptotic either to a plane or to a catenoid. After a rotation of the coordinate system, all these components are graphs of functions of the same two variables.
\end{lconjecture}

\subsection{Main results}\label{sec:mainres}
Our first result is a full parallel to De Giorgi's conjecture in four dimensions, thus breaking the analogy with minimal hypersurfaces.
\begin{theorem}[Classification in $\R^4$]\label{thm:R4indclass}
    Let $u:\R^4\to[-1,1]$ be a finite index solution to the Allen--Cahn equation, with $\mathcal E(u,B_R)\leq CR^3$ for some constant $C$ and all $R>0$. Then $u$ is one-dimensional.\\
     In other words, $u$ is either identically $\pm 1$ or of the form $\tanh\left(\frac{e\cdot x-s_0}{\sqrt{2}}\right)$ for some unit vector $e\in\Sp^{3}$ and $s_0\in\R$.
\end{theorem}
In fact, we obtain a sharp conditional result in dimensions $4\leq n \leq 7$:
\begin{theorem}[Conditional classification]\label{thm:finindclass}
    Fix $n\in\N$ with $4\leq n \leq 7$, and assume that \eqref{eq:stableclass} below is true in $\R^n$. Let $u:\R^n\to[-1,1]$ be a finite index solution to the Allen--Cahn equation, with $\mathcal E(u,B_R)\leq CR^{n-1}$ for some constant $C$ and all $R>0$. Then $u$ is one-dimensional.
    \begin{equation}\label{eq:stableclass}\tag{$\star$}
        \begin{aligned}
        &\text{Let } u:\mathbb{R}^n \to [-1,1] \text{ be a stable solution to the Allen--Cahn equation, with } \mathcal E(u,B_R)\leq CR^{n-1}\\
        & \text{for some constant }C\text{ and all }R>0. \text{ Then } u \text{ is one-dimensional.}
        \end{aligned}
    \end{equation}
\end{theorem}
The classification result in \eqref{eq:stableclass} holds for $n\leq 3$ by \cite{AAC01,AC00}. The recent article \cite{FSstable} by Serra and the author confirms \eqref{eq:stableclass} also for $n=4$, thus showing \Cref{thm:R4indclass} from \Cref{thm:finindclass}.\\

Based on this evidence, we believe it is not far-fetched to propose:
\begin{lconjecture}[Finite index De Giorgi conjecture]\label{conj:indclasnogrowth}
    Let $4\leq n\leq 7$, and let $u:\R^n\to[-1,1]$ be a finite Morse index solution to the Allen--Cahn equation. Then, $u$ is one-dimensional.
\end{lconjecture}
To our knowledge, this conjecture has never been formulated in the literature, unlike more modest variants such as \Cref{conj:indimpends}. We may highlight, nevertheless, two additional existing results which may be seen as compelling evidence towards it, and which are especially relevant to the present paper:
\begin{itemize}
    \item In dimension $3$, many finite index examples of minimal surfaces (and of associated Allen--Cahn solutions \cite{dPKW13, AdPW15, GuiLiuWei}) are known by now. On the other hand, in dimensions $n\geq 4$ there are remarkably few examples of minimal hypersurfaces already; among those with finite index, we highlight the higher-dimensional catenoid (which is axially symmetric) and the Fakhi--Pacard examples \cite{FP00} obtained via gluing methods.
    
    In the phase transition setting, the article \cite{AdPW16} constructs axially symmetric solutions to Allen--Cahn whose zero level set is a logarithmic perturbation of the higher-dimensional catenoid. However, these solutions are shown to have infinite index instead in dimensions $4\leq n\leq 10$.
    
    \item In \cite{GuiWangWei2020}, the authors study  general axially symmetric solutions to the Allen--Cahn equation in $\R^n$, with $3\leq n \leq 10$. Natural examples include then the catenoidal solutions constructed in the articles mentioned above---we recall that, in the case of minimal hypersurfaces, the only complete, connected, embedded minimal hypersurface with axial symmetry is precisely the (generalised) catenoid.
    
    The results in \cite{GuiWangWei2020} confirm that the behaviour found in all such examples extends---in the finite index case---to the class of axially symmetric solutions: Indeed, in the 3D case, the authors show that such solutions only have finitely many ends, and in dimensions $4\leq n\leq 10$ they show that any such solution must be one-dimensional. This confirms, in particular, Conjectures \ref{conj:indimpends} and \ref{conj:indclasnogrowth} up to dimension $10$ in the axially symmetric case. Let us emphasise that, at least for \Cref{conj:indclasnogrowth}, the restriction $n\leq 7$ is necessary in general: for instance, there are energy-minimising (thus, in particular, finite index) solutions to Allen--Cahn in $\R^8$ whose transition level sets are asymptotic to a leaf of the Bombieri--De Giorgi--Giusti foliation of the Simons cone \cite{Simons68, BDGG69}, see \cite{PW13, LWW17b}.
\end{itemize}
{\bf A geometric application}

We record a direct application of \Cref{thm:R4indclass}, and which would seem hopeless in the three-dimensional case. Morally, it says that the transition layers of solutions with bounded energy and index\footnote{As for instance those constructed in \cite{Guaraco2018,GasparGuaraco2018}.} in four dimensions \textit{invariably behave like smooth minimal hypersurfaces}.
\begin{corollary}[Behaviour in closed 4-manifolds]\label{thm:4closedbeh}
    Let $M^4$ be a closed\footnote{Meaning compact and without boundary.}, four-dimensional Riemmanian manifold. Let $u_{\ep}$ be a solution to the $\ep$-Allen--Cahn equation on $M$, with 
    $$\mathcal E^{\ep}(u_{\ep},M):= \frac{1}{\sigma_{n-1}}\int_{M} \left( \frac{\ep}{2}|\nabla u_\ep|^2 + \frac{1}{\ep} W(u_\ep) \right) \leq E_0$$
    and Morse index bounded by $I_0$.
    
    Then, there exists $\ep_0=\ep_0(M,E_0,I_0)>0$ such that: If $\ep\leq \ep_0$, then
    $$\{u_{\ep}=t\}\ \ \mbox{is a smooth hypersurface for every}\ \ |t|\leq 0.9\,.$$
    In fact, there is some $c=c(M,E_0,I_0)>0$ such that $|\nabla u_{\ep}|\geq \frac{c}{\ep}$ in $\{|u_{\ep}|\leq 0.9\}$.

    Additionally, given $\Lambda\geq 1$, up to making $\ep_0$ smaller---depending now also on $\Lambda$---the following hold:
    \begin{itemize}
        \item {\bf Slow curvature degeneration:} We have $\mathcal A_{u_{\ep}}\leq \frac{1}{\Lambda\ep}$ in $\{|u_{\ep}|\leq 0.9\}$.
        \item {\bf Sheet separation:} For every $p\in \{|u_{\ep}|\leq 0.9\}$, letting $t=u_{\ep}(p)$ the set $\{u_{\ep}=t\}\cap B_{\Lambda\ep}(p)$ consists of a single connected graph in normal coordinates.
    \end{itemize}
\end{corollary}
The first bullet above shows that the only possible model of curvature concentration (``bubble''), for a sequence of $u_\ep$ as in the corollary with $\ep\to 0$, is a minimal hypersurface of bounded index in $\R^4$.
It would be interesting to investigate whether further geometric consequences can be derived from this; for instance, a natural question is whether the topology of the level sets can be controlled in terms of the energy and index of the solutions, mirroring \cite{CKM17} or \cite{BS18}.\\

{\bf The three-dimensional case}

As discussed above, the higher-dimensional behaviour from the previous results represents not only a drastic deviation from minimal hypersurface theory, but also from \textit{three-dimensional Allen--Cahn theory itself}---in which a wealth of solutions with finite index and density have been constructed via gluing methods.

In fact, given any complete, nondegenerate, embedded minimal surface in $\R^3$ with finite index, there exists an entire Allen--Cahn solution with the same index and whose zero level set is a logarithmic perturbation of it, by \cite{dPKW13}. In particular, all such examples have finitely many parallel ends, which are asymptotic either to a plane or to a catenoid by construction, leading the authors to put forward \Cref{conj:dPKW} (see also the surveys \cite{dPKW2009survey,dPKW12}).

Our final result confirms this conjecture for solutions with quadratic energy growth, validating the analogy with minimal surfaces and showing a much closer analogue of \Cref{thm:TyskIndex} for $n=3$.
\begin{theorem}[Structure in $\R^3$]\label{thm:R3indclass}
    Let $u:\R^3\to[-1,1]$ be a finite index solution to the Allen--Cahn equation, with $\mathcal E(u,B_R)\leq CR^2$ for some constant $C$ and all $R>0$. Then $u$ has finitely many parallel ends, which are either planar or catenoidal. More precisely, there exist some $R_0>0$ and $N\in\N$, as well as constants $b_i,c_i$ for every $i=1,...,N$, such that: Up to a rotation, we can write
    \begin{equation}
        \{u=0\} \setminus \left(\overline{B_{R_0}'}\times[-R_0,R_0]\right)=\bigcup_{i=1}^N {\rm graph}\, f_i\,,
    \end{equation}
    where $f_1<...<f_N$, 
    $$f_i=b_i+c_i\log |x'|+O(|x'|^{-\alpha})\quad\mbox{for some}\quad \alpha>0\,,$$
    and $|c_{i+1}-c_i|> \sqrt{2}$ for every $i=1,...,N-1$.
\end{theorem}
A well-known related question---see \cite{GuiLiuWei,dPKW12} and conjecture (2) in \cite[p. 127]{dPKW13}---and which we do not address here, is to show that solutions in $\R^3$ with two ``regular'' ends (or with index one) are axially symmetric, by analogy with \cite{Schoen83}. We hope that the results in our article might provide a starting point for its study.

\subsection{Overview of the proofs}\label{sec:overproof}
We focus on explaining why we can conclude Theorems \ref{thm:finindclass} and \ref{thm:R3indclass} despite the {\it optimality} of \Cref{thm:sheetimpc2alpha} as stated.

Consider a stable solution to Allen--Cahn in $B_R\subset\R^n$ with bounded energy density, and assume that its zero level set decomposes into several almost-parallel layers $\Gamma_i$, with
\begin{equation}\label{eq:2738ytgbog}
    |\mathrm{II}_{\Gamma_i}|=O(\frac{1}{R})\quad\mbox{in}\quad B_{R/2}\,;
\end{equation}
this is known to always hold if $n\leq 7$ and \eqref{eq:stableclass} is true in $\R^n$.

An important part of the work in \cite{WW19} consists in showing that, letting $H_i$ denote the mean curvature of $\Gamma_i$, and letting $d_{\Gamma_i,\Gamma_{i+1}}$ denote an appropriate (signed) distance function, one has the Toda-type system
\begin{equation}\label{eq:78gtoqrbg}
    H_i=\kappa_0 \left(e^{-\sqrt{2}d_{\Gamma_i,\Gamma_{i-1}}} - e^{-\sqrt{2}d_{\Gamma_i,\Gamma_{i+1}}}\right)+Err,\quad\mbox{where}\quad Err=O(\frac{1}{R^2})\,.
\end{equation}
This additive error is \textit{optimal} in general. Via sheet separation estimates, the final conclusion in \cite{WW19} is that $e^{-\sqrt{2}d_{\Gamma_i,\Gamma_{i-1}}}+e^{-\sqrt{2}d_{\Gamma_i,\Gamma_{i+1}}}=O(\frac{1}{R^2})$ as well.

Hence, at this point, working with the equation 
\begin{equation}\label{eq:78gtoqrbg2}
    H=O(\frac{1}{R^2})
\end{equation}
instead is just {\it equivalent} to \eqref{eq:78gtoqrbg}, and {\it optimal} as stated. In particular, the interaction information between different components---which we were reading off \eqref{eq:78gtoqrbg}---is lost beyond this point.\\

A main realisation from this is that arguing by using the estimates provided by stability, such as \eqref{eq:78gtoqrbg2}, merely on balls centered away from the origin---as is customary in these kinds of problems---{\bf cannot} possibly allow us then to conclude\footnote{For instance, there are examples of nontrivial minimal hypersurfaces of finite index in $\R^4$, such as the higher-dimensional catenoid.} the surprising rigidity result for finite index solutions in dimensions $4\leq n \leq 7$ (nor the precise asymptotics in \Cref{thm:R3indclass})! With this in mind, here is a brief overview of how we will nevertheless obtain our results:
\begin{itemize}
    \item Let $u:\R^n\to[-1,1]$, $n\leq 7$, be an entire solution with finite index, with $\mathcal E(u,B_R)\leq C R^{n-1}$ for all $R>0$, and assume that \eqref{eq:stableclass} holds. The finite index condition is equivalent to stability far away from the origin, which then leads to \eqref{eq:78gtoqrbg2} for the zero level set of $u$ in balls away from the origin.
    \item Simons' classification \cite{Simons68}---applied to blow-down rescalings of $\{u=0\}$, which are stable minimal cones---yields then flatness on large scales for $\{u=0\}$.
    \item Via an improvement of flatness iteration in annuli (for the PDE \eqref{eq:78gtoqrbg2}), we establish then the uniqueness of such blow-downs, and a decomposition of $\{u=0\}$ (away from a large ball) as a union of graphs of functions $f_1<...<f_N$ over a single hyperplane.
    
    This simple---yet powerful---idea can be applied to a wide range of problems, but we have not found it elsewhere in the literature: indeed, the author first heard about it from X. Fern\'andez-Real and J. Serra in relation to \cite{CFFS25}. For this reason, a general version of this technique will be given in \cite{FFS25A}; still, for the purpose of the present article, we develop this result (with a related, but different, iteration-type proof) for the PDE \eqref{eq:78gtoqrbg2} in the self-contained \Cref{sec:annimprovhyper}.
    \item This iteration gives moreover strong sublinear bounds on the growth of the $f_i$; namely, that the $f_i$ have growth $O(|x|^{\delta})$ for any $\delta>0$. This is essentially optimal just from \eqref{eq:78gtoqrbg2}, yet still far from the desired behaviour. 
    \item However, we can bootstrap this---via the $C^{\alpha}$ mean curvature estimates in \cite{WW19}---to $|\mathrm{II}_{\{u=0\}}|\lesssim|x|^{\delta-2}$, which {\it improves} the curvature estimates from \eqref{eq:2738ytgbog} by almost a full power. Remarkably, this estimate escapes the ``scaling-invariant'' regime.
    \end{itemize}
    All of the above can, however, also be obtained in the minimal hypersurface setting; in fact, with much stronger consequences, since we have been forced to argue with the restrictive condition \eqref{eq:78gtoqrbg2} for now. This is therefore just a starting point, in preparation for the use of \textit{pure Allen--Cahn theory}:
    \begin{itemize}
    \item Revisiting carefully the gluing arguments\footnote{Fortunately for us, a fair number of the modifications required in \cite{WW19,WW18} have already been outlined---in the axially symmetric case---in \cite{GuiWangWei2020}.} in \cite{WW19,WW18}, the improved curvature estimates on $\{u=0\}$ show that $u$ can be better approximated by an ``Ansatz solution''---built as a superposition of 1D Allen--Cahn solutions---than what is optimal in a general, ``scaling-invariant'' situation.
    
    \item The full power of the argumentation in \cite{WW19} leads then in this setting to 
    $$
\Delta f_i\simeq H_i=\kappa_0 \left(e^{-\sqrt{2}(f_i - f_{i-1})} - e^{-\sqrt{2}(f_{i+1} - f_i)}\right)+Err'\,,
$$
where $Err'=O(\frac{1}{|x'|^{2+1/8}})$ is now genuinely of lower order---and not only that, but we obtain moreover an associated {\it stability condition}, also up to a lower order error, for this PDE system. In other words, the \textbf{interaction} between layers becomes leading again. This is the main result in \Cref{sec:ACstruct}.
\end{itemize}
After reaching this point, our goal is now to beat the natural scaling $\Delta f_i\simeq \frac{1}{|x'|^2}$ for solutions of the system above. We argue differently depending on the value of $n$. 
\begin{itemize}
\item In the case $4\leq n\leq 7$, we argue as follows:
\begin{itemize}
    \item It is known that the Liouville equation $\Delta v = e^{-v}$
does not admit solutions in dimensions $3$ to $9$ which are stable outside of a ball \cite{DF09}. Thanks to the graphicality and the improved errors (which are ``lower order''), letting $v=f_{i+1}-f_i$ instead the same proof of Dancer--Farina can be replicated here, giving a contradiction in case there were more than one end (i.e. if $N>1$).
\item Therefore, necessarily $N=1$. On the other hand, this easily implies then that $u$ has energy density $1$ at infinity. We conclude that $u$ must be 1D by Wang's Allard-type theorem (see \cite{Wang17} or \Cref{thm:ACallard}).
\end{itemize} 
\item In the case $n=3$, we argue as follows instead:
\begin{itemize}
    \item Testing stability with a linear cutoff supported in a large annulus $B_R'\setminus \overline{B_{R_0}'}$ and letting $R\to\infty$ gives that
    \begin{equation}\label{eq:412fahpyh4y9uqby}
        \int_{\R^2\setminus \overline{B_{2R_0}'}} e^{-\sqrt{2}(f_{i+1}-f_i)}\leq C\,.
    \end{equation}
    \item Fix $i\in\{1,...,N\}$; by the Toda system we find then that $\Delta f_i\in L^1(\R^2\setminus \overline{B_{2R_0}'})$. This may seem like a really minor improvement from $|\Delta f_i|\leq\frac{C}{|x|^2}$; on the other hand, we have already broken the natural scaling of the equation.
    
    \item We can define then the ``Ansatz'' potential
    $$f_i^a(y):=\frac{1}{2\pi}\int_{\R^{n-1}\setminus \overline{B_{2R_0}'}} (\log|y-z|-\log |z|)\, \Delta f_i(z)\,dz\,.$$
    This ``renormalisation trick'' can be found, for instance, in \cite{CL91}. Letting $ c_i^a:=\frac{1}{2\pi}\int_{\R^{n-1}\setminus \overline{B_{2R_0}'}} \Delta f_i\in\R$, it follows easily that
    $$f_i^a = (c_i^a+o(1))\log|y|\quad\mbox{as}\quad|y|\to\infty\,.$$
    \item Setting $f_i^b:=f_i-f_i^a$, we see then that $\Delta f_i^b= 0$. A priori, $f_i^b$ could be just any harmonic function; on the other hand, it additionally has {\it sublinear growth} (since both $f_i$ and $f_i^a$ do).
    \item An enhanced version of the improvement of flatness iteration in annuli\footnote{Notice that $\Delta v_i^b$ is not limited by a fixed rate (unlike e.g. \eqref{eq:78gtoqrbg2}) anymore.} then shows that actually $$f_i^b=b_i^b+c_i^b\log|y|+O(|y|^{-\beta})\quad\mbox{as}\quad|y|\to\infty$$
    for every $\beta\in(0,1)$.
    \item Since $f_i=f_i^a+f_i^b$, this shows the asymptotics claimed in \Cref{thm:R3indclass} up to improving the expansion for $f_i^a$. For this, we need to go back to studying the interactions between different layers.
    \item As a starting point, we immediately see that $|c_{i+1}-c_i|\geq \sqrt{2}$, just by the integrability of the interactions in \eqref{eq:412fahpyh4y9uqby}. To conclude, we need to show that the inequality is strict: then, the integrability of the interactions improves by an algebraic rate, which easily yields the precise expansion for the $f_i^a$.
    \item Finally, the fact that $|c_{i+1}-c_i|= \sqrt{2}$ can indeed never occur is obtained by a careful inspection of the structure of the Toda system and of the integral defining $f_i^a$ in this setting. A key idea is that the uppermost and lowermost layers are ``special'': they are respectively sub- and superharmonic, since they satisfy (by \eqref{eq:78gtoqrbg}) equations with opposite signs, which in effect ``compresses'' all intermediate layers and ends up preventing the critical interaction.
\end{itemize} 
\end{itemize}

\noindent {\bf Acknowledgements.} It is a pleasure to thank Xavier Fern\'andez-Real and Joaquim Serra, from whom the author first learnt about the idea of improvement of flatness iterations in annuli. The author also thanks Hardy Chan, Marco Guaraco and Kelei Wang for insightful conversations about the results in the article.

The author was supported by the European Research Council under the Grant Agreement No 948029. 

\section{Previous results from the literature}
\subsection{Some general results on critical points}
\begin{lemma}[Modica inequality, \cite{Modica85}]\label{lem:modicaineq}
Let $u:\R^n\to\R$ be a bounded A--C solution on all of $\R^n$. Then, $|u|\leq 1$, and the inequality is strict unless $u\equiv\pm 1$. Moreover,
    $$\frac{|\nabla u|^2(x)}{2}\leq W(u(x))\,.$$
\end{lemma}
Throughout this article, unless otherwise indicated we assume that $u:\R^n\to(-1,1)$ is a solution to A--C on all of $\R^n$, to make use of the Modica inequality. Set
\begin{equation}\label{eq:sig(n-1
)def}
    \sigma_{n-1}:=\omega_{n-1}\int_{-1}^1 \sqrt{2 W(s)}\,ds\,,\quad\mbox{where } \omega_{n-1}\mbox{ is the volume of the unit ball in } \R^{n-1}.
\end{equation}
We define the \emph{energy density} on balls of radius $r > 0$ by
\begin{equation}\label{eq:densitydef}
    \mathbf{M}_r(u) := \frac{1}{r^{n-1}} \mathcal{E}(u, B_r)\,.
\end{equation}
We will more generally denote ${\bf M}_r(u,x_0):={\bf M}_r(u(\cdot-x_0))$, and we omit $u$ from the notation whenever it is clear from the context.
    \begin{remark}\label{rmk:rescdens}
    For a solution $u$ of $\ep$-A--C instead, we naturally set ${\bf M}_r^\ep(u):=\frac{1}{r^{n-1}}\mathcal E^\ep(u,B_r)={\bf M}_{r/\ep}(u(\ep x))$. Unless otherwise stated, we will always work with $\ep=1$.
\end{remark}
\begin{lemma}[Monotonicity formula, \cite{Modica85b}]\label{lem:monformula}
    Let $u:\R^n\to[-1,1]$ be an A--C solution on all of $\R^n$. Then, ${\bf M}_r$ is monotone nondecreasing in $r$. More precisely,
    \begin{equation}\label{eq:monfor}
        \frac{d}{dr}{\bf M}_r(u)=\frac{1}{r^n}\int_{B_r}\frac{1}{\sigma_{n-1}}\left[W(u)-\frac{|\nabla u|^2}{2}\right]\,dx+\frac{1}{r^{n+1}}\int_{\partial B_r}\frac{1}{\sigma_{n-1}}(x\cdot\nabla u)^2\,d\cH^{n-1}(x)\,.
    \end{equation}
\end{lemma}
We say that $u$ has \emph{bounded energy density} if ${\bf M}_\infty(u):=\lim_{r \to \infty} \mathbf{M}_r(u) < \infty$. In particular, the energy growth assumptions in Theorems \ref{thm:R4indclass}, \ref{thm:finindclass} and \ref{thm:R3indclass} amount precisely to having bounded energy density.
\begin{definition}[Monotone 1D solution]
    Define $g:\R\to\R$ by
    \begin{equation}\label{eq:1dsoldef}
        g(x)=\tanh\left(\frac{x}{\sqrt{2}}\right)\,.
    \end{equation}
\end{definition}
Wang developed an analogue of Allard's regularity theory for stationary varifolds in the Allen--Cahn case. In particular, he obtained the following theorem, which allows to give a new proof of Savin's result \cite{Sav09}.
\begin{theorem}[\cite{Wang17}, Allard-type theorem for A--C]\label{thm:ACallard}
    Let $u:\R^n\to [-1,1]$ be a solution to A--C. Then, there is $\delta=\delta(n)>0$ such that if ${\bf M}_\infty(u)\leq 1+\delta$, then $u$ is either $\pm 1$ or of the form $g(a\cdot x+b)$ for some $a\in\Sp^{n-1}$ and $b\in\R$.
\end{theorem}
Throughout the paper we adopt the following notations: Given a point $x \in \mathbb{R}^n$, we write $x' \in \mathbb{R}^{n-1}$ for its first $n-1$ coordinates, so that $x = (x', x_n)$. Moreover, given a set $\Omega \subset \mathbb{R}^{n-1}$ and a function $g : \Omega \to \mathbb{R}$, we denote by ${\rm graph}\, g$ the set of points in $\Omega \times \mathbb{R}$ satisfying $x_n = g(x')$. Set $\mathcal C_R:=B_R'\times [-R,R]$, where $B_R'\subset \R^{n-1}$.\\

Finally, we recall the following standard lemma:
\begin{lemma}[Exponential decay]\label{lem:expdecay}
    Let $u:\R^n\to\R$ be an A--C solution, and assume that
    $$\{|u|\leq 0.85\}\cap \mathcal C_{4R}\subset \{|x_n|\leq \delta R\}\,,$$
    where $\delta \leq \frac{1}{2}$. Then, there are dimensional constants $C$ and $c_0>0$ such that
    \begin{equation}\label{eq:expdecay}
        \sup_{\mathcal C_{2R}\cap\{|x_n|\geq 2\delta R\}}\left[\frac{|\nabla u|^2}{2}+W(u)\right]\leq Ce^{-c_0\delta R}\,.
    \end{equation}
\end{lemma}

\subsection{Curvature estimates for stable solutions}\label{sec:wangwei}
Let
\begin{align}
        \mathcal A^2(u)&=\begin{cases}
            \frac{|D^2u|^2-|\nabla|\nabla u||^2}{|\nabla u|^2}\quad &\mbox{ if } \nabla u\neq 0\\
            0 &\mbox{ otherwise\,;}
        \end{cases}
        \qquad \quad \mbox{and}\qquad \mathcal A(u)=\left(\mathcal A^2(u)\right)^{1/2}. \label{eq:AC2formdef}
    \end{align}
It is easy to see that if $\nabla u(x)\neq 0$ then $\mathcal A^2(u)(x)=|{\rm I\negthinspace I}_{\{u=u(x)\}}|^2(x)+|\nabla^T \log|\nabla u||^2(x)$, where ${\rm I\negthinspace I}_{\{u=u(x)\}}$ is the second fundamental form of the level set $\{u=u(x)\}$ and $\nabla^T$ denotes the gradient in the directions tangent to $\{u=u(x)\}$.

The following (weaker) form of the stability inequality is well known:
\begin{proposition}[Sternberg--Zumbrun inequality, \cite{SZ98}]\label{prop:SZstab}
    Let $u:\Omega\subset\R^n\to \R$ be a stable solution to A--C, and let $\eta\in C_c^1(\Omega)$. Then,
    \begin{equation}\label{eq:SZineq1}
        \int_{\Omega}\mathcal A^2\eta^2\,|\nabla u|^2\leq\int_{\Omega} |\nabla \eta|^2|\nabla u|^2\,.
    \end{equation}
\end{proposition}
\begin{definition}[Sheeting assumptions]\label{def:sheetassump}
Let $u:B_R\to (-1,1)$ be a stable $\ep$-A--C solution. We say that $u$ satisfies the sheeting assumptions in $B_R$ (with constant $C_1>0$) if
\begin{equation}\label{eq:sheetassump}
     \ep|\nabla u|\geq \frac{1}{C_1}\quad \mbox{and} \quad R|\mathcal{A}_u|\leq C_1 \quad \mbox{in } \{|u|\leq 0.9\}\cap B_R.
\end{equation}
\end{definition}

Since $\mathcal A$ controls the curvatures of the level sets, one has the following standard lemma:
\begin{lemma}\label{lem:rjbagbda}
    In the setting of \Cref{def:sheetassump}, assume that $u(0)= t\in [-0.9,0.9]$. There are $\delta\in(0,1)$ and $C_2$, depending on $C_1$ and $n$, such that---after choosing a suitable Euclidean coordinate frame---the following holds:\\
    Let $\{\Gamma^t_i\}_{i=1}^N$ denote the connected components of $\{u=t\}\cap  \mathcal C_{\delta R}$ intersecting $B_{\delta^2 R}$. Then there are $f^t_1<...< f^t_N$, $f^t_i\in C^\infty(B_{\delta R}')$, such that
    \begin{equation*}
        \Gamma_i^t={\rm graph} \,f_i^t\quad \mbox{in }\ \mathcal C_{\delta R}\quad\mbox{and}\quad |D f_i^t|+R|D^2 f_i^t|\leq C_2\,.
    \end{equation*}
\end{lemma}
Wang and Wei showed that a-priori $C^2$ bounds can be upgraded to $C^{2,\vartheta}$ bounds, obtaining moreover improved separation and mean curvature estimates.
\begin{theorem}[$C^2$ implies $C^{2,\vartheta}$, \cite{WW19}]\label{thm:sheetimpc2alpha}
    Let $u:B_R\subset \R^n\to (-1,1)$ be a stable $\ep$-A--C solution, and assume the sheeting assumptions hold in $B_R$ for some $C_1$. In particular, the conclusion of \Cref{lem:rjbagbda} holds.
    
     Then, for every $\vartheta\in(0,1)$ there are $C_2$ and $\delta_0>0$ depending additionally on $\vartheta$ such that if $\frac{\ep}{R}\leq \delta_0$, then
    \begin{equation}\label{eq:c2alphasheet}
        \|D^2 f_i^t\|_{L^\infty}+[D^2 f_i^t]_{C^\vartheta}\leq \frac{C_2}{R}\,.
    \end{equation}
    Furthermore, letting $\HH[f]:={\rm div}(\frac{\nabla f}{\sqrt{1+|\nabla f|^2}})$ denote the mean curvature operator,
    \begin{equation}\label{eq:meancurvsheet}
         \|\HH[f_i^t]\|_{L^\infty}+ \big[\HH[ f_i^t]\big]_{C^\vartheta}\leq \frac{\ep C_2}{R^2}\,.
    \end{equation}

    Additionally, the separation between layers satisfies
    \begin{equation}\label{eq:wwsepbnd}
        f_{i+1}^t-f_i^t\geq \frac{1+\vartheta}{2}\sqrt{2}\ep\log{\left(\frac{R}{\ep}\right)}.
    \end{equation}
\end{theorem}
\begin{remark}
    We emphasise that we will consider $\ep=1$ in the vast majority of the article.
\end{remark}
\Cref{thm:sheetimpc2alpha} is an a-priori result. However, \textit{assuming} additionally \eqref{eq:stableclass}, local sheeting assumptions for stable solutions hold indeed. It is worth emphasising that, unlike for minimal hypersurfaces, this fact is highly nontrivial and its proof requires using \Cref{thm:sheetimpc2alpha} itself.
\begin{theorem}[Stability conjecture and sheeting]\label{thm:sheetstabconj}
    Fix $n\in\N$, with $n\leq 7$. Assume that \eqref{eq:stableclass} is true in $\R^n$. Then, the following holds:

    Let $u:B_{R}\to(-1,1)$ be a stable $\ep$-A--C solution, with ${\bf M}_{R}\leq C_0$. Then, there are $C_1$ and $\delta_0>0$, depending only on $C_0$, such that if $\frac{\ep}{R}\leq \delta_0$ then $u$ satisfies the sheeting assumptions in $B_{R/2}$ with constant $C_1$.
\end{theorem}
\begin{proof}
    See the proofs of \cite[Theorem 1.4]{FSstable} or \cite[Corollary 1.3]{WW19}.
\end{proof}

\subsection{Wang--Wei in a flat setting}
For convenience, we rewrite \Cref{thm:sheetimpc2alpha} in a flat, large-scale setting.
\begin{theorem}[\cite{WW19} in a flat setting]\label{thm:sheetpart2}
     Let $n\leq 10$. Let $u:B_R\to (-1,1)$ be a stable A--C solution, and assume the sheeting assumptions hold in $B_R$ for some $C_1$. Let $\delta_1>0$ and $\theta\in(0,1)$. Then, there exist $C_2, R_0$ and $c_0>0$, depending on $C_1,\theta$, such that the following hold.\\
    Assume additionally that
    $$\{u=0\}\subset \{|e_n\cdot x|\leq \delta_2 R\}\quad\mbox{in}\quad B_R$$ 
    for some $\delta_2\in(0,c_0)$ and $R\geq R_0$. Then, there is some $N\in\N$ such that:
    \begin{itemize}
        \item There are $C^{\infty}$ graphs $f_i:B_{\frac{3}{5}R}'\to[-R/8,R/8]$, $f_1<...<f_N$, such that $\{u=0\}\cap \mathcal C_{\frac{3}{5}R}=\bigcup_{i=1}^N {\rm graph} \,f_i$.
    \item The estimates \eqref{eq:c2alphasheet}--\eqref{eq:wwsepbnd} hold for the $f_i$.
    \end{itemize}
\end{theorem}
\begin{proof}
    Immediate from \Cref{lem:rjbagbda} and \Cref{thm:sheetimpc2alpha}, by making $\delta_2>0$ sufficiently small.
\end{proof}
We also collect some useful additional properties.
\begin{lemma}[Additional properties]\label{lem:addproperties}
    Let $\delta_1>0$. In the setting of \Cref{thm:sheetpart2}, up to making $\delta_2>0$ smaller and $R_0$ larger depending also on $\delta_1$, the following additionally hold:
    \begin{itemize}
    \item The graphs satisfy $\|\nabla f_i\|_{L^\infty}\leq \delta_1$.
    \item There holds $\{|u|\leq 0.9\}\subset \{|e_n\cdot x|\leq \delta_2 R +C_2\}$ in $\mathcal C_{R/2}$.
    \item We have $|{\bf M}_{R/2}(u)-N|\leq \delta_1 N$.
    \end{itemize}
\end{lemma}
\begin{proof}
See \cite[Lemma 4.9]{FSstable} for the short proof.
\end{proof}

\section{Annular improvement of flatness for hypersurfaces}\label{sec:annimprovhyper}
This section builds an improvement of flatness result for the PDE \eqref{eq:78gtoqrbg2}. As described in \Cref{sec:overproof}, a general version for minimal hypersurfaces is developed in \cite{FFS25A}; we develop here a related yet different proof strategy, based on an iterative statement on nested exterior domains (see \Cref{prop:mainiteration}).

\subsection{Setting and main result}
To simplify the notation, the general framework we assume in the whole of the present section is the following:
\begin{definition}[Standing assumptions]\label{def:standassumpt} We are given:
 \begin{itemize}
     \item An integer $n\geq 3$.
     \item Positive constants $C_0,R_0$.
     \item A nonempty, embedded hypersurface $\Sigma\subset \R^n\setminus \overline{B_{R_0/2}}$, with $\bar \Sigma \setminus \Sigma\subset \partial B_{R_0/2}$, such that
         \begin{equation}\label{eq:89tqrbnnbs}
        |x||\mathrm{II}_\Sigma|+ |x|^2|\mathrm{H}_\Sigma|\leq C_0\,.
    \end{equation}
 \end{itemize}
\end{definition}
Given $E\subset \R^n$ and $e\in\Sp^{n-1}$, $\mathcal R_e(E)$ will denote the image of $E$ under the rotation bringing $e_n$ to $e$.

Our main result in this section is the following:
\begin{proposition}[Main result]\label{prop:dec}
There exists some $\ep_0>0$, depending on $n,C_0$, such that the following holds. Assume that for every $R\geq 8R_0$ there is some $e_R\in\Sp^{n-1}$ with
$$\Sigma \cap (B_{8R}\setminus \overline{B_{R/8}})\subset \{|e_R\cdot x|\leq \ep_0 R\}\,.$$
Then, there is some $\widetilde R_0\geq 16R_0$, depending on $n,C_0,R_0$, as well as some $e_\infty\in\Sp^{n-1}$, $N\in\N$, and smooth functions $f_i:\R^{n-1}\setminus \overline{B_{\widetilde R_0}}\to\R$ for $i=1,...,N$, such that
$$
\mathcal R_{e_\infty}(\Sigma) \setminus \left(\overline{B_{\widetilde R_0}'}\times[-\widetilde R_0,\widetilde R_0]\right)=\bigcup_{i=1}^N {\rm graph}\, f_i\,.
$$
Additionally, given $\alpha\in(0,1)$, there is $C_\alpha$ depending on $n,C_0,R_0,\alpha$ such that
\begin{equation}\label{eq:tyahtsh9yhw}
     |f_i(x')|+|x'||\nabla f_i|\leq C_\alpha|x'|^\alpha\,.
\end{equation}
\end{proposition}

\subsection{Reduction to an iterative result}
We start with a simple lemma that says that having controlled curvatures and small height in a larger annulus implies graphicality over a smaller annulus. We will use this repeatedly.
\begin{lemma}[Graphicality lemma]\label{lem:mf3rt7qwg0tba}
    Let $1<R_1<R_2<R_3$ and $C_0$. Then, there is $\ep_0=\ep_0(R_1,R_2,R_3,C_0,n)>0$, with $R_2+\ep_0< R_3$, such that the following holds.
    
    Assume that $\widetilde\Sigma$ is an embedded hypersurface with $(\overline{\widetilde \Sigma}\setminus \widetilde \Sigma) \cap (B_{R_3}\setminus \overline{B_{R_3^{-1}}})=\emptyset$,
    $$|\mathrm{II}_{\widetilde\Sigma}|\leq C_0\quad\mbox{in}\quad B_{R_3}\setminus \overline{B_{R_3^{-1}}}\,,$$
    and
    $$
    \widetilde\Sigma\cap (B_{R_3}\setminus \overline{B_{R_3^{-1}}})\subset\{|x_n|\leq \ep_0\}\,.
    $$
    Then
    $$\widetilde\Sigma \cap (B_{R_1}\setminus \overline{B_{R_1^{-1}}})\subset (B_{R_2}'\setminus \overline{B_{R_2^{-1}}'})\times[-\ep_0,\ep_0]\,,$$
    and there are some $N\in\N$ and some smooth $f_i:B_{R_2}'\setminus \overline{B_{R_2^{-1}}'}\to [-\ep_0,\ep_0]$ such that
    $$\widetilde\Sigma \cap  \left((B_{R_2}'\overline{\setminus B_{R_2^{-1}}'})\times[-\ep_0,\ep_0]\right)=\bigcup_{i=1}^N {\rm graph}\, f_i\,.$$
\end{lemma}
\begin{proof}
{\bf Step 1.} Inclusion.
    
We can directly see that $\widetilde\Sigma \cap (B_{R_1}\setminus \overline{B_{R_1^{-1}}})\subset (B_{R_2}'\setminus \overline{B_{R_2^{-1}}'})\times [-\ep_0,\ep_0]$, just by taking
$$\ep_0\leq \min\{\sqrt{R_1^{-2}-R_2^{-2}},\frac{1}{2}(R_3-R_2)\}\,.$$
Indeed, if $x\in \widetilde\Sigma\cap (B_{R_1}\setminus \overline{B_{R_1^{-1}}})$ then $|x|\geq R_1^{-1}$, and moreover $|x_n|\leq \ep_0\leq \sqrt{R_1^{-2}-R_2^{-2}}$ by assumption; we can then compute 
$$|x'|=\sqrt{|x|^2 - |x_n|^2}\geq \sqrt{R_1^{-2}-(\sqrt{R_1^{-2}-R_2^{-2}})^2}=R_2^{-1}\,.$$
{\bf Step 2.} Graphicality.

    Exactly as in \Cref{lem:rjbagbda}, by the second fundamental form bound there is $c_0=c_0(R_1,R_2,R_3,C_0,n)>0$, with $R_2+\ep_0+2c_0< R_3$ and $R_2^{-1}-2c_0> R_3^{-1}$, such that:
    Given $x\in \widetilde\Sigma \cap (B_{R_2}'\setminus \overline{B_{R_2^{-1}}'})\times[-\ep_0,\ep_0]$---so that $B_{2c_0}(x)\subset (B_{R_3}\setminus \overline{B_{R_3^{-1}}})$---we can write all of the components of $\widetilde\Sigma$ intersecting $B_{c_0}(x)$ as a finite union of graphs with respect to some direction $e_x$. Note that the finiteness just comes from the fact that having infinitely many components locally would contradict embeddedness.
    
    Now, by the $\ep_0$-flatness assumption, up to making $\ep_0>0$ small enough depending on $c_0$ we immediately see that we can actually take $e_x=e_n$ above. Moreover, just by interpolation, we can make the graphs have slope less than any $\delta>0$ arbitrary up to making $\ep_0$ smaller.
    
    Take the union over $x\in \widetilde\Sigma \cap (B_{R_2}'\setminus \overline{B_{R_2^{-1}}'})\times[-\ep_0,\ep_0]$ of these graphs. Considering the maximal domains of definition $\Omega_i\subset (B_{R_2}'\setminus \overline{B_{R_2^{-1}}'})$ we obtain associated $f_i:\Omega_i\to [-\ep_0,\ep_0]$ such that
    $$\widetilde\Sigma\cap \left((B_{R_2}'\setminus \overline{B_{R_2^{-1}}'})\times[-\ep_0,\ep_0]\right)=\bigcup_{i=1}^{N}{\rm graph}\,f_i\,.$$
    
    We claim that $\Omega_i=(B_{R_2}'\setminus \overline{B_{R_2^{-1}}'})$ for each $i=1,...,N$. But indeed, this just follows from connectedness of $(B_{R_2}'\setminus \overline{B_{R_2^{-1}}'})$: the maximal set of definition of each $f_i$ in $(B_{R_2}'\setminus \overline{B_{R_2^{-1}}'})$ is both open (by embeddedness of $\Sigma$ and the fact that $|\nabla f_i|\leq \delta$) and closed (since $(\overline\Sigma\setminus \Sigma)\cap (B_{R_3}\setminus \overline{B_{R_3^{-1}}})=\emptyset$).
\end{proof}

We can immediately deduce the following modest but important step, which says that $\Sigma$ cannot ``fold'' back and forth between dyadic scales:
\begin{lemma}[$\ep_0$-flatness implies finite spiral ends]\label{lem:spiralstruc}
    There exists some $\ep_0>0$, depending on $n,C_0$, such that the following holds. Assume that for every $R\geq 8R_0$ there is some $e_R\in\Sp^{n-1}$ with
    $$\Sigma \cap (B_{8R}\setminus \overline{B_{R/8}})\subset \{|e_R\cdot x|\leq \ep_0 R\}\,.$$

    Then, the following hold:
    \begin{itemize}
        \item {\bf Finiteness of ends:} $\Sigma \cap (\R^n \setminus \overline{B_{16R_0}})=\bigcup_{i=1}^N \Sigma_i$ for some $N\in\N$, where the $\Sigma_i$ are \textbf{connected}.
        \item {\bf Spiral structure:} For every $R\geq 16R_0$, $\Sigma_i \cap \partial B_R$ is nonempty. Moreover, $\Sigma_i \cap (B_{2R}\setminus \overline{B_{R/2}})$ is a \textbf{single} graphical sheet, in the sense that:
        \begin{itemize}
            \item $\mathcal R_{e_R}(\Sigma_i) \cap (B_{2R}\setminus \overline{B_{R/2}})\subset (B_{2R}'\setminus \overline{B_{R/4}'})\times [-\ep_0R,\ep_0R]$.
            \item There is some smooth $f_i:B_{2R}'\setminus \overline{B_{R/4}'}\to [-\ep_0R,\ep_0R]$ such that
            $$\mathcal R_{e_R}(\Sigma_i) \cap \left((B_{2R}'\setminus \overline{B_{R/4}'})\times[-\ep_0R,\ep_0R]\right)={\rm graph}\, f_i\,.$$
        \end{itemize}
    \end{itemize}
\end{lemma}
\begin{proof}
    We directly see that
    $$\mathcal R_{e_R}(\Sigma) \cap (B_{2R}\setminus \overline{B_{R/2}})\subset (B_{2R}'\setminus \overline{B_{R/4}'})\times [-\ep_0R,\ep_0R]\,,$$
    just by considering $\widetilde \Sigma_R:=\frac{1}{R}\mathcal R_{e_R}(\Sigma)$, applying \Cref{lem:mf3rt7qwg0tba} (for instance with $R_1=2$, $R_2=4$, $R_3=8$, and using \eqref{eq:89tqrbnnbs} to bound $|\mathrm{II}_{\widetilde \Sigma_R}|\leq 8C_0$ in $B_8\setminus B_{1/8}$), and then scaling back.

    Moreover---assuming that $\Sigma \cap (B_{2R}\setminus \overline{B_{R/2}})\neq \emptyset$---this also gives some $N_R\in\N$, which is finite but which \textit{may depend on $R$} at this point, and smooth $f_i:B_{2R}'\setminus \overline{B_{R/4}'}\to [-\ep_0 R,\ep_0R]$ for every $i=1,..,N_R$, such that
    $$\mathcal R_{e_R}(\Sigma) \cap \left((B_{2R}'\setminus \overline{B_{R/4}'})\times[-\ep_0R,\ep_0R]\right)=\bigcup_{i=1}^{N_R} {\rm graph}\, f_i\,.$$

    It remains to connect the different scales. If $\Sigma\setminus \overline{B_{16R_0}}$ were empty the Lemma is trivially true; assume therefore that $\Sigma\setminus \overline{B_{16R_0}}\neq\emptyset$, so that in particular there is some $\bar R\geq 16 R_0$ such that $\Sigma\cap \partial B_{\bar R}\neq \emptyset$. We can now see that in fact $\Sigma\cap \partial B_R$ is nonempty for \textit{every} $R\geq 16R_0$:
    
    Assume for contradiction that this failed for some $R>\bar R$, and let $R^*$ be the infimum among such radii. Since $\bar \Sigma\setminus\Sigma\subset \partial B_{R_0/2}$, we see that $\Sigma\cap \partial B_{R^*}\neq\emptyset$. By our argument above,
    $$\Sigma\cap \partial B_{R^*}\subset \Sigma\cap(B_{2R_*}\setminus \overline{B_{R^*/2}})\subset (B_{2R^*}'\setminus \overline{B_{R^*/4}'})\times[-2\ep_0R^*,2\ep_0R^*]\,,$$
    and we can write
    $$\mathcal R_{e_{R^*}}\Sigma\cap \left((B_{2R^*}'\setminus \overline{B_{R^*/4}'})\times[-2\ep_0R^*,2\ep_0R^*]\right)$$ as a union of graphs over $B_{2R^*}'\setminus \overline{B_{R^*/4}'}$. But such graphs are connected, and they trivially also contain some $x\in \Sigma$ with $|x|\geq 2R^*$. This shows that $\Sigma\cap \partial B_R\neq\emptyset$ for all $R\in[R^*,2R^*]$, contradiction the definition of $R^*$ as an infimum. The same proof works---up to making $\ep_0$ small enough---for $16 R_0\leq R<\bar R$, up to considering a supremum instead.\\

    Finally, observe that we naturally must have \textit{matching decompositions} on any pair of overlapping, local graphical decompositions, which allows us to identify different sheets across scales; this shows that the $N_R$ above is continuous in $R$, and therefore constant and equal to some $N\in\N$. We can then identify each of the local graphs as pieces of some globally defined connected components $\{\Sigma_i\}_{i=1}^N$. The two bullets in the lemma follow directly from our previous discussion.
\end{proof}
For simplicity, we \textit{assume that $\Sigma$ is connected} from now on, which we can do by \Cref{lem:spiralstruc} up to considering one of the $\Sigma_i$ and $16R_0$ instead of $\Sigma$ and $R_0$.

We do not yet know the global graphicality of $\Sigma$ away from a ball, as the $e_R$ could rotate for now; we therefore work with extrinsic statements until the end. \Cref{prop:dec} will follow in a straightforward manner from the following:
\begin{proposition}[$\ep_0$-flatness implies decay]\label{prop:dec2}
Given $\alpha\in(0,1)$, there exists some $\ep_0>0$, depending on $n,C_0,R_0,\alpha$, such that the following holds. Assume that for every $k\in\N$ with $2^{k+1}\geq R_0$ there is some $e_k\in\Sp^{n-1}$ with
$$\Sigma \cap (B_{2^{k+3}}\setminus \overline{B_{2^{k-3}}})\subset \{|e_k\cdot x|\leq \ep_0 2^k\}\,.$$
Then, there are $C>0$ and some $\widetilde e_k\in\Sp^{n-1}$ for every $k\in\N$ such that
\begin{align*}
    \Sigma \cap (B_{2^{k+3}}\setminus \overline{B_{2^{k-3}}})&\subset \{|\widetilde e_k\cdot x|\leq C\ep_0 2^{\alpha k}\}\,.
\end{align*}
\end{proposition}
Now, to obtain this result, we will in turn reduce it to another extrinsic statement, which is the technical core of the proof: an ``improvement of flatness''-type iteration. Finding the right statement is rather tricky, as one needs to update different information about different scales at each step.
\begin{proposition}[Main iteration]\label{prop:mainiteration}
    Given $\alpha\in(0,1)$, there are $k_0,K\in\N$, and $\ep_0>0$, all depending on $n$ and $\alpha$ such that the following holds. Set $\delta=2^{\alpha-1}$. Assume we are given:
    \begin{itemize}
        \item Some base scale $k_1\in\N$, with $k_1\geq K$.
        \item Some ``best flatness'' $\ep>0$, with $2^{(\alpha-1) k_1}\leq \ep\leq \ep_0$.
        \item Some vectors $e_k\in \Sp^{n-1}$ for all $k\geq k_1-k_0$, such that the following hold:
\begin{itemize}
    \item Excess decay between $k_1-k_0$ and $k_1$:
    $$|x\cdot e_{k}|\leq \delta^{k-k_1}\ep 2^{k} \mbox{ for all } x\in\Sigma\cap(B_{2^{k+1}}\setminus \overline{B_{2^{k-1}}}),\,\quad k_1-k_0\leq k \leq k_1\,.$$
    \item Preservation of best flatness: $$|x\cdot e_{k}|\leq \ep 2^{k} \mbox{ for all } x\in\Sigma\cap(B_{2^{k+1}}\setminus \overline{B_{2^{k-1}}}),\,\quad k_1+1\leq k \,.$$
\end{itemize}
    \end{itemize}
Then, we have the following:
    \begin{itemize}
        \item Set the new base scale $k_1':=k_1+1\in\N$, which satisfies $k_1'\geq K$.
        \item Set the new ``best flatness'' $\ep'=\delta\ep>0$, which satisfies $2^{(\alpha-1) k_1'}\leq \ep'\leq \ep_0$.
        \item Then, there are some new vectors $e_k'\in \Sp^{n-1}$ for all $k\geq k_1'-k_0$, such that the following hold:
\begin{itemize}
    \item Excess decay between $k_1'-k_0$ and $k_1'$:
    $$|x\cdot e_{k}'|\leq \delta^{k-k_1'}\ep' 2^{k} \mbox{ for all } x\in\Sigma\cap(B_{2^{k+1}}\setminus \overline{B_{2^{k-1}}}),\,\quad k_1'-k_0\leq k \leq k_1'\,.$$
    \item Preservation of best flatness: $$|x\cdot e_{k}'|\leq \ep' 2^{k} \mbox{ for all } x\in\Sigma\cap(B_{2^{k+1}}\setminus \overline{B_{2^{k-1}}}),\,\quad k_1'+1\leq k \,.$$
\end{itemize}
    \end{itemize}
\end{proposition}
\begin{remark}
    In words, every time we have a large block of scales where flatness decays, and assuming moreover that the best flatness then remains for all scales up to infinity, then flatness decays up to an extra scale, and moreover this improved flatness also propagates to all further scales up to infinity. In particular, if the assumptions of \Cref{prop:mainiteration} hold for some initial flatness and scale, then it can actually be iterated as many times as wanted.
\end{remark}
\Cref{prop:mainiteration} will be proved in the next section. For now, we show how \Cref{prop:dec2} is implied by \Cref{prop:mainiteration}, and how \Cref{prop:dec} is implied by \Cref{prop:dec2}.
\begin{proof}[Proof of \Cref{prop:dec2}, assuming \Cref{prop:mainiteration}]
    Fix $\alpha\in (0,1)$. Let $k_0,K\in\N$ and $\ep_0>0$ be given by \Cref{prop:mainiteration}, and set $\delta=2^{\alpha-1}$. By assumption (in \Cref{prop:dec2}), we have
    \begin{equation}\label{eq:67ehjkbgavb}
        \Sigma \cap (B_{2^{k+3}}\setminus \overline{B_{2^{k-3}}})\subset \{|e_k\cdot x|\leq \ep_0 2^k\}\quad\mbox{for every}\quad k\geq \lceil \log_2{R_0}\rceil\,.
    \end{equation}
    It is clear now how to proceed:
    \begin{itemize}
        \item Set $k_1:=\max\{K,k_0+\lceil \log_2{R_0}\rceil\}$, so that $k_1\geq K$.
        \item Up to making $k_1$ larger, we also have $2^{(\alpha-1)k_1}\leq \ep_0$.
        \item Since (by definition) $k_1-k_0\geq \lceil \log_2{R_0}\rceil$, by \eqref{eq:67ehjkbgavb} we have
        $$|x\cdot e_k|\leq \ep_02^k\quad \mbox{for all}\quad x\in\Sigma\cap(B_{2^{k+1}}\setminus \overline{B_{2^{k-1}}})\quad\mbox{and}\quad k\geq k_1-k_0\,.$$
        \begin{itemize}
            \item Since this includes the range $k_1+1\leq k$, we have the ``preservation of best flatness'' condition.
            \item Moreover, since this includes the range $k_1-k_0\leq k\leq k_1$, we obviously also have the ``excess decay'' condition, i.e. $|x\cdot e_k|\leq \delta^{k-k_1}\ep_02^k$, simply since $\delta^{k-k_1}\geq 1$ in this range.
        \end{itemize}
        
    \end{itemize}
    This means that the hypotheses of \Cref{prop:mainiteration} are satisfied. Iterating the proposition, setting $k_{1,l}:=k_1+l$ and $\ep_l:=\delta^l \ep_0$, $l\in\N$, the ``preservation of best flatness'' part gives---at the $l$-th iteration---a vector $\widetilde e_{k_{1,l}}$ such that
    $$|x\cdot \widetilde e_{k_{1,l}}|\leq \ep_l2^{k_{1,l}}= \delta^l \ep_0 2^{k_{1,l}}\quad\mbox{for all}\quad x\in \Sigma\cap (B_{2^{k_{1,l}+1}}\setminus \overline{B_{2^{k_{1,l}-1}}})\,.$$
    Letting $C:=\delta^{-k_1}\ep_0$, which does not depend on $l$, we can rewrite this as
    $$|x\cdot \widetilde e_{k_{1,l}}|\leq  C\delta^{k_{1,l}} 2^{k_{1,l}}\quad\mbox{for all}\quad x\in \Sigma\cap (B_{2^{k_{1,l}+1}}\setminus \overline{B_{2^{k_{1,l}-1}}})\,,$$
    or (since $\delta=2^{\alpha-1}$)
    $$|x\cdot \widetilde e_{k_{1,l}}|\leq  C 2^{\alpha k_{1,l}}\quad\mbox{for all}\quad x\in \Sigma\cap (B_{2^{k_{1,l}+1}}\setminus \overline{B_{2^{k_{1,l}-1}}})\,,$$
    which clearly proves \Cref{prop:dec2}.
\end{proof}
\begin{proof}[Proof of \Cref{prop:dec}, assuming \Cref{prop:dec2}]
    By assumption, there are $\widetilde e_k\in\Sp^{n-1}$ such that for every $k\in\N$ with $2^k\geq 8R_0$, we have
\begin{align}\label{eq:37y4tgweeub}
    \Sigma \cap (B_{2^{k+3}}\setminus \overline{B_{2^{k-3}}})&\subset \{|\widetilde e_k\cdot x|\leq C\ep_0 2^{\alpha k}\}\,.
\end{align}

Recall that, by \Cref{lem:spiralstruc}, up to arguing with each $\Sigma_i$ separately it suffices to prove \Cref{prop:dec} in the case in which $\Sigma$ is a single connected piece, i.e. satisfying the second bullet there.

A remark is in order: To be precise, this will show the existence of some $e_{\infty,i}$ satisfying the thesis of \Cref{prop:dec} with $\Sigma_i$ in place of $\Sigma$, with the only caveat that $e_{\infty,i}$ could a priori depend on $i$. However, we then immediately see that the $e_{\infty,i}$ need in fact to be all equal just by embeddedness.\\

\noindent {\bf Step 1.} Summing the coefficients.

Applying \eqref{eq:37y4tgweeub} with $k$ and $k+1$ we deduce in particular that
\begin{align}\label{eq:t2831t41gtq}
    \Sigma \cap (B_{2^{k+1}}\setminus \overline{B_{2^{k-1}}})&\subset \{|\widetilde e_k\cdot x|\leq C\ep_0 2^{\alpha k}\}\cap \{|\widetilde e_{k+1}\cdot x|\leq C\ep_0 2^{\alpha (k+1)}\}\,.
\end{align}
Then,  the local graphicality (coming from \Cref{lem:spiralstruc}) of $\Sigma$ and the triangle inequality easily give that
$$|\widetilde e_{k+1}-\widetilde e_k|\leq \frac{C 2^{\alpha k}}{2^k}=C2^{(\alpha-1)k}\,.$$
Then, given $k_0\in \N$ with $2^{k_0}\geq 8R_0$ we find that
$$\sum_{k\geq k_0} |\widetilde e_{k+1}-\widetilde e_k|\leq \sum_{k\geq k_0} C2^{(\alpha-1)k} \leq C_\alpha 2^{(\alpha-1)k_0}\,,$$
and in particular the sequence $\widetilde e_k$ is Cauchy. This shows the existence of some $ e_\infty\in \Sp^{n-1}$ with
\begin{equation}\label{eq:q7toagvg}
    | e_\infty-\widetilde e_{k_0}|\leq \sum_{k\geq k_0} |\widetilde e_{k+1}-\widetilde e_k|\leq C_\alpha 2^{(\alpha-1)k_0}\,,\quad\mbox{thus}\quad \Sigma \cap (B_{2^{k+3}}\setminus \overline{B_{2^{k-3}}})\subset \{| e_\infty\cdot x|\leq CC_\alpha 2^{\alpha k}\}\,.
\end{equation}
In other words, $\{e_\infty\cdot x=0\}$ is the (unique) blow-down tangent of $\Sigma$.\\

\noindent {\bf Step 2.} Graphicality.

Thanks to \eqref{eq:q7toagvg}, we can now just apply \Cref{lem:spiralstruc} but with the constant direction $\widetilde e_\infty$ (in place of the $e_R$). This shows that $\mathcal R_{e_\infty}(\Sigma)\setminus \left(\overline{B_{\widetilde R_0}'}\times[-\widetilde R_0,\widetilde R_0]\right)$ is the graph of some $f:\R^{n-1}\setminus \overline{B_{\widetilde R_0}'}\to\R$. Moreover, combining \eqref{eq:37y4tgweeub} and \eqref{eq:q7toagvg} we see that
\begin{align}\label{eq:89ygoiubjbad}
    \mathcal R_{ e_\infty}(\Sigma) \cap (B_{2^{k+3}}\setminus \overline{B_{2^{k-3}}})&\subset \{|x_n|\leq CC_\alpha 2^{\alpha k}\}\,,\quad\mbox{thus}\quad |f(x')|\leq CC_\alpha |x'|^\alpha.
\end{align}
Finally, combining \eqref{eq:89ygoiubjbad} and standard elliptic estimates for the minimal graph equation (using $|x|^2|H_\Sigma|\leq C_0$) we can upgrade \eqref{eq:89ygoiubjbad} to
$$|f|+|x'||\nabla f|\leq C_\alpha|x'|^\alpha$$
as desired.
\end{proof}

\subsection{Establishing the iterative result --- Proof of \Cref{prop:mainiteration}}
We need the following simple lemma for a function which is almost harmonic on a large annulus and has controlled growth.
\begin{lemma}[Harmonic approximation]\label{lincomp} Fix $\lambda, \alpha\in (0,1)$. There exists $\delta_0\in(0,\lambda]$, depending on $\lambda, \alpha, n$, such that the following holds. Let $v:B_{\frac{1}{\delta_0}}\setminus \overline{B_{\delta_0}}\subset \R^{n-1} \to \R$ satisfy:
\begin{itemize}
    \item $|{\rm Tr}(A(x')D^2 v)|\leq \delta_0$, where $\|A-{\rm Id}\|_{L^\infty(B_{\frac{1}{\delta_0}}\setminus \overline{B_{\delta_0}})}\leq \delta_0$.
    
    \item $\|v\|_{L^\infty(B_{2\rho}\setminus \overline{B_\rho})}\leq \rho^{2-\alpha}$, for $1\leq\rho\leq\frac{1}{\delta_0}$.
    
    \item $\|v\|_{L^\infty(B_{2\rho}\setminus \overline{B_\rho})}\leq \rho^{\alpha}$, for $\delta_0\leq\rho\leq 1$.
\end{itemize}
Then there exists $a\in\R^{n-1}$ such that $\|v-a\cdot x\|_{L^\infty(B_{\frac{1}{\lambda}}\setminus \overline{B_{\lambda}})
}\leq \lambda$.
\end{lemma}
\begin{proof}
We argue by contradiction. If no such $\delta_0>0$ exists, considering the sequence $\delta_{0,i}=\frac{1}{i}\to 0$ there need to be associated $v_i$ as in the hypotheses of the theorem but such that
\begin{equation}\label{eq:71401gfvpiwvb}
    \mbox{no}\ \ a\in\Sp^{n-1}\ \ \mbox{satisfies}\quad\|v_i-a\cdot x\|_{L^\infty(B_{2}\setminus \overline{B_{1/2}})
}\leq \lambda\quad\mbox{for any}\ \  i \ \ \mbox{in the sequence.}
\end{equation}
For $i$ large enough, the $v_i$ satisfy $C^{1,\alpha}_{loc}(\R^2\setminus\{0\})$ bounds, by the standard Cordes--Nirenberg estimates \cite{Cordes1956, Nirenberg1954}. By Ascoli--Arzel\`a, a subsequence converges locally to some limit $v_\infty$ in $C^{1,\alpha}$, which is a---viscosity, thanks to the $C^0_{loc}$ convergence, and hence strong---solution to the Laplace equation on $\R^{n-1}\setminus\{0\}$. Moreover, $v_\infty$ grows less than $|x|^\alpha$ as $|x|\to 0$ and less than $|x|^{2-\alpha}$ as $|x|\to\infty$.

By the usual Liouville theorem for harmonic functions in $\R^n\setminus \{0\}$, which follows for instance by separation of variables, the growth assumptions imply that $v_\infty=a\cdot x$ for some $a\in\Sp^{n-1}$. Indeed, the fundamental solution of the Laplacian in $\R^{n-1}$ cannot be a term in $v_\infty$ since it blows up at $0$, and quadratic or higher order terms cannot contribute to $v_\infty$ either since they grow faster than $|x|^{2-\alpha}$ at infinity.

But then, the uniform convergence in $B_2\setminus \overline{B_{1/2}}$ to $v_\infty$ gives a contradiction with \eqref{eq:71401gfvpiwvb} for $\delta_{0,i}$ small enough.
\end{proof}
{\bf Preliminaries: Rescaling and local graphicality}

For convenience, we start by considering $\widetilde \Sigma:=\frac{1}{2^{k_1}}\Sigma$.
\begin{remark}\label{rmk:89ttqovba}
    The original hypotheses---in scales \textbf{from $-k_0$ to $\infty$} now---become:
    \begin{itemize}
    \item $|x||\mathrm{II}_{\widetilde \Sigma}|\leq C$
    \item $|x|^2|H[\widetilde \Sigma]|\leq 2^{-k_1}$
    
    \item We have
    \begin{itemize}
    \item Excess decay between $-k_0$ and $0$:
    \begin{equation}\label{eq:s1r138tqhq}
        |x\cdot \widetilde e_{k}|\leq \delta^{k}\ep 2^{k} \mbox{ for all } x\in\widetilde \Sigma\cap(B_{2^{k+1}}\setminus \overline{B_{2^{k-1}}}),\,\quad -k_0\leq k \leq 0\,.
    \end{equation}
    \item Preservation of best flatness:
    \begin{equation}\label{eq:s1r138tqhq2}
        |x\cdot \widetilde e_{k}|\leq \ep 2^{k} \mbox{ for all } x\in\widetilde \Sigma\cap(B_{2^{k+1}}\setminus \overline{B_{2^{k-1}}}),\,\quad 1\leq k \,.
    \end{equation}
\end{itemize}
\end{itemize}
\end{remark}

\begin{lemma}[Local graphicality lemma]\label{lem:grphregnHyp}
    There exist some $c_0>0$ and $C_2$, depending only on $n$, $k_0$ and $C_0$, such that the following holds. Assume that $\ep\leq c_0$. Then, there is some $f:B_{2^{k_0-1}}'\setminus \overline{B_{2^{-k_0+1}}'}\to\R$ such that
    \begin{equation}\label{eq:894tibhabHyp}
        \mathcal R_{\widetilde e_0}(\widetilde \Sigma)\cap \left((B_{2^{k_0-1}}'\setminus \overline{B_{2^{-k_0+1}}'})\times [-2^{k_0-1},2^{k_0-1}]\right)={\rm graph}\, f\,,
    \end{equation}
with
$$f(x')\leq C_2\ep|x'|^{\alpha}\quad\mbox{for}\quad 2^{-k_0+1}\leq|x'|\leq 1$$
and
$$f(x')\leq C_2\ep|x'|\log(|x'|+1)\quad\mbox{for}\quad 1\leq|x'|\leq 2^{k_0-1}\,.$$
Moreover,
$$\mathcal R_{\widetilde e_0}(\widetilde \Sigma) \cap (B_{2^{k_0-1}}\setminus \overline{B_{2^{-k_0+2}}})\subset (B_{2^{k_0-1}}'\setminus \overline{B_{2^{-k_0+1}}'})\times [-2^{k_0-1},2^{k_0-1}]\,.$$
\end{lemma}
\begin{proof}
    Without loss of generality $\widetilde e_0=e_n$.

    We first consider $-k_0+1\leq k\leq 0$. Applying \eqref{eq:s1r138tqhq} with $(k-1)$ and $k$, we deduce in particular that
\begin{align}\label{eq:t2831t41gtq}
    \widetilde \Sigma \cap (B_{2^{k}}\setminus \overline{B_{2^{k-1}}})&\subset \{|\widetilde e_{k-1}\cdot x|\leq \delta^{k-1}\ep 2^{k-1}\}\cap \{|\widetilde e_{k}\cdot x|\leq \delta^k\ep 2^k\}\,.
\end{align}
Then,  the local graphicality (coming from \Cref{lem:spiralstruc}) of $\widetilde \Sigma$ and the triangle inequality easily give that
$$
|\widetilde e_{k+1}-\widetilde e_{k}|\leq \frac{C\delta^k\ep 2^k}{2^k}=C\delta^k\ep\,,\quad -k_0\leq k\leq 0\,.
$$
Similarly one obtains
$$
|\widetilde e_{k+1}-\widetilde e_{k}|\leq \frac{C\ep 2^k}{2^k}=C\ep\,,\quad 1\leq k\,.
$$
Then, using that $|x\cdot\widetilde e_0|\leq |x\cdot \widetilde e_k|+|x||\widetilde e_k-\widetilde e_0|$, summing the corresponding geometric series to bound $|\widetilde e_k-\widetilde e_0|$ we deduce that
    \begin{equation}\label{eq:7g51qy4tq08v}
    |x\cdot \widetilde e_0|\leq C\delta^{k}\ep2^k\quad \mbox{for all}\quad x\in\widetilde \Sigma\cap(B_{2^{k}}\setminus \overline{B_{2^{k-1}}})\quad\mbox{and}\quad -k_0\leq k \leq 0\,,
\end{equation}
    \begin{equation}\label{eq:7g51qy4tq08v2}
    |x\cdot \widetilde e_0|\leq C\ep k2^{k}\quad\mbox{for all}\quad x\in\widetilde \Sigma\cap(B_{2^{k}}\setminus \overline{B_{2^{k-1}}})\quad\mbox{and}\quad 1\leq k\,.
\end{equation}
Having fixed $k_0$ and all other parameters, the graphicality claim \eqref{eq:894tibhabHyp} and the corresponding inclusion are proved just via \Cref{lem:mf3rt7qwg0tba}, up to making $\ep_0>0$ small enough (recall that $\ep\leq \ep_0$ by assumption). Finally, \eqref{eq:7g51qy4tq08v}--\eqref{eq:7g51qy4tq08v2} give exactly the growth conditions for $f$.
\end{proof}
{\bf Proof of the iteration}

We can finally give:
\begin{proof}[Proof of \Cref{prop:mainiteration}]
As above, we consider $\widetilde \Sigma:=\frac{1}{2^{k_1}}\Sigma$ in place of $\Sigma$, which satisfies the hypotheses in \Cref{rmk:89ttqovba}. After a rotation, we moreover can (and do) assume that for $k=0$ we have $\widetilde e_{0}=e_n$.

By \Cref{lem:grphregnHyp}, as long as $\ep_0>0$ is small enough there is some $f:B_{2^{k_0-1}}'\setminus \overline{B_{2^{-k_0+1}}'}\to\R$ such that
    \begin{equation}\label{eq:894tibhabHyp2}
        \widetilde \Sigma\cap \left((B_{2^{k_0-1}}'\setminus \overline{B_{2^{-k_0+1}}'})\times [-2^{k_0-1},2^{k_0-1}]\right)={\rm graph}\, f\,,
    \end{equation}
with
\begin{equation}\label{eq:89ygqiubrgipab}
    f(x')\leq C_2\ep|x'|^{\alpha}\quad\mbox{for}\quad 2^{-k_0+1}\leq|x'|\leq 1
\end{equation}
and
\begin{equation}\label{eq:89ygqiubrgipab2}
    f(x')\leq C_2\ep|x'|\log |x'|\leq C_2\ep|x'|^{2-\alpha}\quad\mbox{for}\quad 1\leq|x'|\leq 2^{k_0-1}\,.
\end{equation}
Let $\lambda>0$, the (universal) choice of which will be given at the end. Let $\delta_0=\delta_0(\lambda,\alpha,\beta,n)\in(0,\lambda]$ be given by \Cref{lincomp}.\\
We argue in several steps.

\noindent {\bf Step 1.} Linearisation.

By the uniform curvature bound assumption, up to making $\ep_0$ smaller (in terms of $\delta_0,n,k_0,C_0$), a standard interpolation with \eqref{eq:89ygqiubrgipab}--\eqref{eq:89ygqiubrgipab2} shows that $|\nabla f|\leq \delta_0$ in $2^{-k_0+2}\leq|x'|\leq 2^{k_0-2}$. Note that
$$H[\widetilde \Sigma](x',f(x'))={\rm div}\left(\frac{\nabla f}{\sqrt{1+|\nabla f|^2}}\right)(x')={\rm Tr}(A(x')D^2f)\,,$$
where $A(x')$ is the $n\times n$ matrix with components
$$A_{ij}(x')=\frac{\delta_{ij}}{\sqrt{1+|\nabla f|^2}}-\frac{\partial_i f \partial_j f}{(1+|\nabla f|^2)^{3/2}}\,.$$
Up to making $\ep_0$ even smaller, by the bound on $\nabla f$ we see that
$$|A(x')-{\rm Id}|\leq\delta_0 \quad\mbox{in}\quad 2^{-k_0+2}\leq|x'|\leq 2^{k_0-2}\,.$$
Define the vertical scaling $v(x')=\frac{f(x')}{C_2\ep}$, so that (by the estimate for $H[\widetilde \Sigma]$ in \Cref{rmk:89ttqovba}) we get
$$\left|{\rm Tr}(A(x')D^2v)\right|=\left|\frac{H[\widetilde \Sigma](x',f(x'))}{C_2\ep}\right|\leq C\frac{2^{-k_1}}{2^{-2 k_0}\ep}\leq C(k_0)\frac{2^{-k_1}}{\ep}\,.$$
Using the assumption $\ep\geq 2^{(\beta-1) k_1}$, we can then bound $$\left|{\rm Tr}(A(x')D^2v)\right|\leq C(k_0)2^{-\beta k_1}\,.$$
We want to apply \Cref{lincomp} to $v$. Observe that, since $v(x')=\frac{f(x')}{C_2\ep}$, \eqref{eq:89ygqiubrgipab}--\eqref{eq:89ygqiubrgipab2} transform into
\begin{equation}\label{eq:89ygqiubrgipabv}
    v(x')\leq |x'|^{\alpha}\quad\mbox{for}\quad 2^{-k_0+1}\leq|x'|\leq 1
\end{equation}
and
\begin{equation}\label{eq:89ygqiubrgipab2v}
    v(x')\leq |x'|^{2-\alpha}\quad\mbox{for}\quad 1\leq|x'|\leq 2^{k_0-1}\,.
\end{equation}

Set $k_0:=\lceil|\log_2 \delta_0|\rceil+2$; in particular, $2^{-k_0+2}\leq \delta_0\leq \frac{1}{\delta_0}\leq 2^{k_0-2}$. Choose then $K$ large enough---depending also on $k_0$---so that $C(k_0)2^{-\beta K}\leq \delta_0$, so that (since $k_1\geq K$ by assumption) we obtain
\begin{equation}\label{eq:67fhablsa}
    \left|{\rm Tr}(A(x')D^2v)\right|\leq \delta_0\,.
\end{equation}
Combining all of the above we see that the hypotheses of \Cref{lincomp} are satisfied, hence we get $a\in\R^{n-1}$ such that $\|v-a\cdot x'\|_{L^\infty(B_{\frac{1}{\lambda}}\setminus \overline{B_{\lambda}})
}\leq \lambda$. Equivalently:
\begin{equation}\label{eq:14t8ggabg}
    \|f-\bar a\cdot x'\|_{L^\infty(B_{\frac{1}{\lambda}}\setminus \overline{B_{\lambda}})
}\leq C_2\ep\lambda\quad\mbox{for some}\quad \bar a\in\R^{n-1}.
\end{equation}

\noindent {\bf Step 2.} Updating the coefficients.\\
We will choose $\lambda\leq 1/4$, thus (by the choices in Step 1) also $\delta_0\leq \lambda\leq 1/4$ and $k_0=\lceil|\log_2 \delta_0|\rceil+2\geq 4$.

\noindent {\bf Definition of $e_{k_1'}'$:} Since (by \eqref{eq:89ygqiubrgipab}--\eqref{eq:89ygqiubrgipab2}) we have $|f(x')|\leq 16C_2\ep$ in $\frac{1}{4}\leq |x'|\leq 4$, we immediately see that $|\bar a|\leq C_3\ep$. Let $e_{k_1'}':=\frac{(-\bar a,1)}{|(-\bar a,1)|}$.

Now, by \Cref{lem:grphregnHyp}, we have $\widetilde \Sigma \cap (B_{4}\setminus \overline{B_{1/4}})\subset {\rm graph}\, f$. But then, we can bound
\begin{equation}\label{eq:8t1oibwg6d}
    |x\cdot e_{k_1'}'|=|(x',f(x'))\cdot\frac{(-\bar a,1)}{|(-\bar a,1)|}|\leq |f(x')-\bar a\cdot x'|\leq  C_2\ep\lambda\quad\mbox{for any}\quad x\in \widetilde \Sigma \cap (B_{4}\setminus \overline{B_{1/4}})\,.
\end{equation}

\noindent {\bf Step 3.} Excess decay between $k_1'-k_0$ and $k_1'$.

Recall that we set $k_1'=k_1+1$ and $\ep'=\delta \ep$. Choosing finally $\lambda=\min\{2\delta C_2^{-1},1/4\}$, and recalling that $\Sigma=2^{k_1}\widetilde \Sigma$, we find that
\begin{equation}\label{eq:214thbqiobvap}
    |x\cdot e'_{k_1'}|\leq (\ep 2\delta) 2^{k_1}=\ep'2^{k_1'} \mbox{ for all } x\in\Sigma\cap(B_{2^{k_1'+1}}\setminus \overline{B_{2^{k_1'-1}}})\,.
\end{equation}
Now, for $k\leq k_1'-1$ instead, we choose to just keep $e_k':=e_k$. Using \eqref{eq:214thbqiobvap} for $k=k_1'$, and simply using the excess decay assumption in \Cref{prop:mainiteration} for scales $k_1'-k_0\leq k\leq k_1'-1$, we find that
$$|x\cdot e_{k}'|\leq C'\delta^{k-k_1'}\ep'2^{k} \mbox{ for all } x\in\Sigma\cap(B_{2^{k+1}}\setminus \overline{B_{2^{k-1}}}),\, k_1'-k_0\leq k \leq k_1'\,.$$
This means precisely that we have shown the excess decay between scales $k_1'-k_0$ and $k_1'$, as desired.

\noindent {\bf Step 4.} Preservation of best flatness.

Take any $k_1''\geq k_1+1$.

\noindent {\bf Claim.} The hypotheses of \Cref{prop:mainiteration} hold with the same $\ep$ and $e_k$ but with $k_1''$ in place of $k_1$.

Indeed:
\begin{itemize}
        \item We have $k_1''\geq k_1 +1\geq K$.
        \item We also have $2^{(\alpha-1) k_1''}\leq 2^{(\alpha-1) k_1}\leq\ep\leq \ep_0$.
        \item The same $e_k\in \Sp^{n-1}$, with $k\geq k_1''-k_0$ now, satisfy:
\begin{itemize}
    \item Excess decay between $k_1''-k_0$ and $k_1''$:
    $$|x\cdot e_{k}|\leq \delta^{k-k_1}\ep 2^{k}\leq \delta^{k-k_1''}\ep 2^{k} \mbox{ for all } x\in\Sigma\cap(B_{2^{k+1}}\setminus \overline{B_{2^{k-1}}}),\,\quad k_1''-k_0\leq k \leq k_1''\,.$$
    \item Preservation of best flatness: $$|x\cdot e_{k}|\leq \ep 2^{k} \mbox{ for all } x\in\Sigma\cap(B_{2^{k+1}}\setminus \overline{B_{2^{k-1}}}),\,\quad k_1''+1\leq k \,.$$
\end{itemize}
\end{itemize}
But then, this means that we can apply Steps 1 to 3 above again, just with $k_1''$ in place of $k_1$. The conclusion of Step 3 gives then the existence of some new $e_{k_1''+1}'\in\Sp^{n-1}$ such that
$$|x\cdot e'_{k_1''+1}|\leq \ep'2^{k_1''+1} \mbox{ for all } x\in\Sigma\cap(B_{2^{k_1''+1}}\setminus \overline{B_{2^{k_1''-1}}})\,.$$
Since $k_1''\geq k_1+1$ was arbitrary, this gives precisely the preservation of the best flatness $\ep'$ that we needed to conclude the proof of \Cref{prop:mainiteration}.
\end{proof}

\section{Structure of finite index Allen--Cahn solutions}\label{sec:ACstruct}
\subsection{Graphicality}\label{sec:graphicality}
The main result of this section is:
\begin{proposition}\label{prop:endgrowth}
    Let $3\leq n\leq 7$. Let $u:\R^n\to[-1,1]$ be a finite index solution to A--C, with bounded energy density. Assume that \eqref{eq:stableclass} is true in $\R^n$.
    
    Then, there are some $R_0>0, e\in\Sp^{n-1}$ and $N\in\N$, all depending on $u$, such that the following holds:
    
    There exist $N$ smooth functions $f_i:\R^{n-1}\setminus \overline{B_{R_0}'}\to\R$, $f_1<...<f_N$, such that
    \begin{equation*}
        \mathcal R_e\left(\{u=0\}\right)=\bigcup_i {\rm graph}\, f_i\quad \mbox{in}\quad \R^{n}\setminus (\overline{B_{R_0}'}\times [-R_0,R_0])\,.
    \end{equation*}
    Additionally, given $\alpha>0$, there is $C$ depending also on $\alpha$ such that
    \begin{equation}\label{eq:endgrowth}
         |f_i| +|x||D f_i|+|x|^2 |D^2 f_i|\leq C|x|^{\alpha}\,.
    \end{equation}
\end{proposition}
In \Cref{sec:todaimprov} we will then use the second-order flatness of $\{u=0\}$---given by \eqref{eq:endgrowth}---to show that $u$ itself is well approximated by a superposition of 1D Allen--Cahn solutions, leading to an elliptic system (Toda) for the $f_i$ with small error.\\

\noindent {\bf Allen--Cahn setting.}
\begin{lemma}
In the setting of \Cref{prop:endgrowth}, there are $C_1$ and $R_0$, as well as a modulus of continuity $\omega$, such that the following hold:
\begin{itemize}
    \item Stability.
    $$u \mbox{ is a stable solution to A--C in } \R^n\setminus \overline{B_{R_0}}\,.$$
    \item Density bound. \begin{equation}\label{eq:engrowth}
            {\bf M}_R(u)\leq C_1\qquad\mbox{for every}\quad R>0\,.
        \end{equation}
    \item Sheeting assumptions. \begin{equation}\label{eq:sheetgb}
        |\mathcal{A}|(x)\leq \frac{C_1}{|x|}\quad \mbox{and}\quad |\nabla u|\geq \frac{1}{C_1}\quad \mbox{in}\quad \{|u|\leq 0.9\}\setminus \overline{B_{R_0}}\,.
        \end{equation}
    \item Large-scale flatness. \begin{equation}\label{eq:flatrot1}
        \mathcal R_{e_R}(\{u=0\})\cap (B_{8R}\setminus \overline{B_{R/8}})\subset\{|x_n|\leq \omega(R^{-1}) R\}\quad \mbox{for some}\quad e_R\in\Sp^{n-1},\quad\mbox{given}\quad R\geq R_0\,.
        \end{equation}
    \item Density lower bound. \begin{equation}\label{eq:nodensgaps}
        {\bf M}_\infty-\omega(R^{-1})\leq {\bf M}_{\omega(R^{-1})R}(y)\leq {\bf M}_\infty \quad\mbox{for all}\quad y\in \{e_R\cdot x=0\}\cap (B_{8R}\setminus \overline{B_{R/8}}),\quad\mbox{with } e_R\mbox{ from } \eqref{eq:flatrot1}.
        \end{equation}
\end{itemize}
\end{lemma}
\begin{proof}
The finite index condition implies that there is some $R_0$ such that $u$ is stable in $\R^n\setminus \overline{B_{R_0}}$: Otherwise, for every $i\in\N$ we would find some $\eta_i\in C_c^1(\R^n\setminus \overline{B_{2^i}})$ such that, in the notation of \eqref{eq:2vardef}, $Q(\eta_i,\eta_i)<0$. Then---since the $\eta_i$ have compact support---we would find an infinite subsequence $\eta_{i_k}$ such that all of their supports are disjoint and $Q(\eta_{i_k},\eta_{i_k})<0$. This easily shows that $Q(\eta,\eta)<0$ for any $\eta$ obtained as a nontrivial, finite linear combination of the $\eta_{i_k}$, giving a contradiction.

The second bullet holds since ${\bf M}_\infty(u)<\infty$ and we have the monotonicity formula (\Cref{lem:monformula}).

Then, \Cref{thm:sheetstabconj} shows that the sheeting assumptions hold in any large enough ball away from the origin, which---up to making $R_0$ larger and choosing $C_1$ accordingly---gives the third bullet.\\

    The large-scale flatness property \eqref{eq:flatrot1} is obtained by arguing by contradiction: Assume the existence of a sequence $R_k\to\infty$ such that \eqref{eq:flatrot1} does not hold for $R=R_k$. Consider $\widetilde u_k(x):=u(R_k x)$, which are solutions to $\ep_k$-A--C with $\ep_k=\frac{1}{R_k}$. By \Cref{thm:sheetimpc2alpha}, applied away from the origin, the level sets $\{\widetilde u_k=0\}$ satisfy---locally---uniform  area and $C^{2,\alpha}$ estimates. Then, the Arzel\`a--Ascoli theorem gives their subsequential, local $C^2$ convergence in $\R^n\setminus\{0\}$ to a limit submanifold $\widetilde\Sigma_\infty$, which is minimal by \eqref{eq:meancurvsheet}. By monotonicity of ${\bf M}_R(u)$, we easily see that $\widetilde\Sigma_\infty$ has constant area density equal to ${\bf M}_\infty(u)$, thus it is a cone. Moreover, passing \eqref{eq:SZineq1} to the limit, we deduce that it is actually stable away from the origin. Then, Simons' classification \cite{Simons68} shows that $\widetilde\Sigma_\infty$ is actually a hyperplane; since the $\{\widetilde u_k=0\}$ converge in $B_{16}\setminus \overline{B_{1/16}}$ to $\widetilde\Sigma_\infty$, letting $e_{R_k}$ be the normal vector to this hyperplane and scaling back we reach a contradiction.

    Finally, \eqref{eq:nodensgaps} follows by the same contradiction argument plus a simple comparison; see the proof of \cite[Proposition 3.10]{FSstable} for full details.
\end{proof}

Let $\{\Sigma_i\}_{i\in I}$ be an enumeration of the (possibly countably infinite) connected components of $\{u=0\}$ in $\R^n\setminus \overline{B_{R_0/2}}$, and let $\Sigma:=\bigcup_{i\in I} \Sigma_i$.
\begin{lemma}\label{lem:9017oqiubg}
Assume that ${\bf M}_\infty(u)> 0$. Then, $\Sigma$ satisfies the assumptions in \Cref{def:standassumpt}. Moreover, given $\ep_0>0$, up to making $R_0$ larger the following holds: Given $R\geq R_0$, there is some $e_R\in\Sp^{n-1}$ with
\begin{equation}\label{eq:1673gpbrgq}
    \Sigma \cap (B_{8R}\setminus \overline{B_{R/8}})\subset \{|e_R\cdot x|\leq \ep_0 R\}\,.
\end{equation}
\end{lemma}
\begin{proof}
    By the implicit function theorem, which we can apply due to \eqref{eq:sheetgb}, $\Sigma$ is relatively closed in $\R^n\setminus \overline{B_{R_0/2}}$ and embedded.

    Moreover, since $u$ satisfies the sheeting assumptions (again by \eqref{eq:sheetgb}), \Cref{thm:sheetimpc2alpha} gives some $C_0$ such that the (second fundamental form and) mean curvature bounds in \eqref{eq:89tqrbnnbs} hold for $\Sigma$.

    The nonemptiness is immediate from, for example, a simple application of the third bullet in \Cref{lem:addproperties} together with the assumption ${\bf M}_\infty(u)>0$.

    Finally, \eqref{eq:1673gpbrgq} follows directly from \eqref{eq:flatrot1}.
\end{proof}

\begin{proof}[Proof of \Cref{prop:endgrowth}]
    If ${\bf M}_\infty(u)=0$ then $u\equiv\pm 1$, by the monotonicity formula, and the statement is trivial. Otherwise, by \Cref{lem:9017oqiubg} we can apply \Cref{prop:dec} to $\Sigma$. We deduce that the $\Sigma_i$ are a finite collection $\Sigma_1,...,\Sigma_N$, and there is some fixed $e\in\Sp^{n-1}$ such that the graphicality claim holds. Moreover, given $\alpha>0$ there is $C$ depending also on $\alpha$ such that
    \begin{equation*}
         |f_i(x')|+|x'||D f_i| \leq C|x'|^{\alpha}\,.
    \end{equation*}
    The second derivative bound then follows from this, via standard Schauder-type bootstrap estimates for the minimal graph equation plus the H\"older bound for $\mathrm{H}[f_i]$ given by \Cref{thm:sheetimpc2alpha} (see \eqref{eq:meancurvsheet}).
\end{proof}
\begin{remark}\label{rmk:inftydens}
    Combining \Cref{lem:addproperties} and \eqref{eq:nodensgaps} also shows that $N$ is finite; in fact, together with the sublinear growth in \eqref{eq:endgrowth}, they give that $N={\bf M}_\infty(u)$.
\end{remark}

\subsection{Improvement of gluing approximation}\label{sec:todaimprov}
In all of this section we assume that we are in the setting of \Cref{prop:endgrowth}. We remark that the constant $R_0>0$ will be made as large as needed and possibly changing from line to line.

{\bf Main result.}
This section is devoted to showing (see \cite{WW19,GuiWangWei2020} for the precise value of the universal constant $\kappa_0$):
\begin{proposition}[Toda system and stability]\label{prop:todaimprov}
The $\{f_i\}_{i=1}^N$ from \Cref{prop:endgrowth} solve the Toda-type system
\begin{equation}\label{eq:todaimprov}
\Delta f_{i} = \kappa_0 \left(e^{-\sqrt{2}(f_i - f_{i-1})} - e^{-\sqrt{2}(f_{i+1} - f_i)}\right) + O\left(|x'|^{-2 - \frac{1}{8}}\right)\quad\mbox{in}\quad \R^{n-1}\setminus \overline{B_{R_0}'}\,,
\end{equation}
where we interpret $e^{-\sqrt{2}(f_1 - f_{0})}\equiv0$ and $e^{-\sqrt{2}(f_{N+1} - f_{N})}\equiv 0$.

Moreover, they satisfy the following associated stability condition:
Let $i\in\{2,...,N\}$ and $\eta \in C_0^\infty(\R^{n-1}\setminus \overline{B_{R_0}'})$. Then,
\begin{equation}\label{eq:todastabimprov}
2\sqrt{2} \kappa_0 \int_{\R^{n-1}\setminus \overline{B_{R_0}'}} e^{-\sqrt{2}(f_i - f_{i-1})} \eta^2\, dx'\leq \left(1 + CR_0^{-\frac{1}{8}}\right) \int_{\R^{n-1}\setminus \overline{B_{R_0}'}} |\nabla \eta|^2\,dx' + C \int_{\R^{n-1}\setminus \overline{B_{R_0}'}} \eta^2 |x'|^{-2-\frac{1}{8}} \,dx'.
\end{equation}
\end{proposition}
{\bf Improved curvature estimates.} Let $\mathrm{II}_i$ denote the second fundamental form of $\Gamma_i:={\rm graph}\, f_i$.
From \Cref{prop:endgrowth}, it follows in particular\footnote{In fact, we get that $|\mathrm{II}_i(x')| \leq C_\alpha|x'|^{-2+\alpha}$ for any $\alpha>0$, but we work with the fixed power $|x'|^{-3/2}$ in what follows for concreteness.} that
\begin{equation}\label{eq:iashgabg}
|\mathrm{II}_i(x')| \leq C|x'|^{-3/2} \quad \forall |x'| \geq R_0\,,
\end{equation}
In other words, we have improved the Wang--Wei curvature bound (which had a power $-1$) to a stronger decay, breaking scaling invariance. In fact, we can upgrade this to the same bound for the full A--C second fundamental form and higher derivatives as well:
\begin{lemma}\label{lem:2ffdec} We have
$\mathcal A+|\nabla \mathcal A|\leq C|x|^{-3/2}$ in $\{|u|\leq 0.9\}\setminus\overline{B_{R_0}}$, and in particular
$|\mathrm{II}_i|+|\nabla \mathrm{II}_i|\leq C|x|^{-3/2}$.
\end{lemma}
Here $\nabla \mathcal A$ denotes the usual gradient of the function $\mathcal A$ (recall \eqref{eq:AC2formdef}), and $\nabla \mathrm{II}_i$ denotes the intrinsic covariant derivative of the tensor $\mathrm{II}_i$ over the hypersurface $\Gamma_i$ where it is defined.
\begin{proof} 
\noindent {\bf Step 1.} Allen--Cahn second fundamental form.\\
    We first upgrade \eqref{eq:iashgabg} to a bound on the A--C second fundamental form, i.e.
    \begin{equation}\label{eq:ehqbqbajf}
        \mathcal A\leq C|x|^{-3/2}\quad\mbox{in}\quad \{|u|\leq 0.9\}\setminus \overline{B_{R_0}}\,.
    \end{equation}
    Let $x\in \{|u|\leq 0.9\}\setminus \overline{B_{R_0}}$, and let $R=|x|/2$.  We argue in the fixed ball $B_R(x)$, in which---up to making $R_0$ larger--- we know \eqref{eq:iashgabg}.
    
    Now, by the estimates in \cite{WW19} (using \textbf{only} the bound \eqref{eq:sheetgb}, i.e. the sheeting assumptions, for now) we know that $u=g_*+\phi$, where $g_*$ is an Ansatz approximate solution and 
    $$\|\phi\|_{C^{2,\alpha}}\leq C|x|^{-2}\,.$$
    
    We can now compute $\nabla u$ and $\nabla^2 u$ on the transition layers $\{|u|\leq 0.9\}$, by arguing as in\footnote{We follow the arXiv version here instead of \cite{WW18} since it contains full details for this argument.} the two-dimensional case in \cite[page 45]{WW17arXiv}. To be precise, consider Fermi coordinates with respect to the layers $\Gamma_i$, exactly as in \cite{WW17arXiv}, so that $y_k$ (with $k=1,...,n-1$) parametrise one of the $\Gamma_i$, and $z$ parametrises the normal direction to $\Gamma_i$ instead---we note that, for convenience of the reader, the entire Fermi coordinate setting will be reworked and reviewed in full detail after the proof of this Lemma. Up to considering $\frac{\partial}{\partial y_k}$ for $k=1,...,n-1$ in place of $\frac{\partial}{\partial y}$ (since we are working in arbitrary dimension) in the computation in \cite[page 45]{WW17arXiv}, we get
    \begin{equation*}
        \nabla u = (-1)^\alpha g_i' \frac{\partial}{\partial z} +O(|x|^{-2})\quad\mbox{and}\quad \nabla^2 u =g_i'' \frac{\partial}{\partial z}\otimes\frac{\partial}{\partial z} +(-1)^i g_i' \nabla\frac{\partial}{\partial z} +O(|x|^{-2})\,,
    \end{equation*}
    where $g_i'$ and $g_i''$ are suitable translations of the derivatives of the one-dimensional Allen--Cahn solution $g$, essentially centered at the $\Gamma_i$.

    Now, thanks to the improved bound in \eqref{eq:iashgabg}, with the corresponding notation from \cite{WW17arXiv} we can estimate
    $$
    |\mathrm{II}_i(y,z)|\leq C|\mathrm{II}_i(y,0)|\leq C|x|^{-3/2} \qquad\mbox{and}\qquad \left|\nabla\frac{\partial}{\partial z}\right|\leq C |\mathrm{II}_i(y,z)|\leq C|x|^{-3/2}\,.
    $$
    This is the only real modification with respect to the argument in \cite{WW17arXiv}, which just uses the scale-invariant estimate $|\mathrm{II}_i(y,0)|\leq C|y|^{-1}$ instead.
    
    We then find
    \begin{equation*}
        \nabla^2 u = g_i'' \frac{\partial}{\partial z}\otimes\frac{\partial}{\partial z} +O(|x|^{-3/2})\,,
    \end{equation*}
    which by the above leads to
    \begin{equation*}
        \mathcal A^2=\frac{|\nabla^2 u|^2-|\nabla|\nabla u||^2}{|\nabla u|^2}=O(|x|^{-3})
    \end{equation*}
    as desired.
    
    {\bf Step 2.} Higher order bound.\\
    We argue similarly to \cite[Lemma 8.1]{WW18}. By \eqref{eq:sheetgb}, \( \nu = \frac{\nabla u}{|\nabla u|} \) is well defined and smooth in \( \{ |u| \leq 0.9 \} \setminus \overline{B_{R_0}}\). Define \( B := \nabla \nu \), i.e. the differential matrix of the vector field $\nu$; with this notation, we have then $|B|=\mathcal A$.

We want to find an equation satisfied by $B$, in order to upgrade the estimate from \eqref{eq:ehqbqbajf} to higher derivatives as well. Now, by direct calculation, we have
\begin{equation}\label{eq:div2fundf}
    -\operatorname{div}(|\nabla u|^2 \nabla \nu) = |\nabla u|^2 |\nabla \nu|^2 \nu\,.
\end{equation}
Differentiating the equation gives the following Simons-type equation:
\begin{equation}\label{eq:simeq2ff}
-\operatorname{div}(|\nabla u|^2 \nabla B) = |\nabla u|^2 |B|^2 B + |\nabla u|^2 \nabla |B|^2 \otimes \nu + |B|^2 \nabla |\nabla u|^2 \otimes \nu + |\nabla u|^2 \nabla^2 \log |\nabla u|^2 \cdot B.    
\end{equation}
Now, observe that $|\nabla u|$ has positive lower and upper bounds in $\{|u|\leq 0.9\}\setminus \overline{B_{R_0}}$, as well as uniformly bounded derivatives of any order. We immediately see then that all the terms in the right hand side of \eqref{eq:simeq2ff} can be bounded (in norm) by $O(|B|)=O(\mathcal A)$, thus (by \eqref{eq:ehqbqbajf}) we have
\begin{equation}\label{eq:simeq2ff2}
|\operatorname{div}(|\nabla u|^2 \nabla B)| \leq C|x|^{-3/2}\quad\mbox{in}\quad \{|u|\leq 0.9\}\setminus \overline{B_{R_0}}\,.    
\end{equation}
Combining \eqref{eq:ehqbqbajf} and \eqref{eq:simeq2ff2}, standard interior gradient estimates---on balls of small but fixed radius---give then that
\begin{equation}\label{eq:simeq2ff3}
|\nabla B| \leq C|x|^{-3/2}\quad\mbox{in}\quad \{u=0\}\setminus \overline{B_{2R_0}}
\end{equation}
as well, up to making $R_0$ large enough.

Finally, to bound $|\nabla \mathrm{II}|$ instead, it suffices to realise that
\begin{equation}\label{eq:719475giuog}
    |\nabla \mathrm{II}|\leq |\nabla B|+|\mathrm{II}||B|\,,
\end{equation}
and then combine this with \eqref{eq:simeq2ff3} and the bound from Step 1.

To be precise, here $\nabla \mathrm{II}$ denotes the intrinsic covariant derivative of the tensor $\mathrm{II}$ over a level set, and $\nabla B$ denotes the Euclidean differential of $B$, just as before.  One has then---as long as $\nabla u$ does not vanish---the formula
$$
(\nabla_Z \mathrm{II})(X,Y)=(\nabla_Z B)(X,Y)+\mathrm{II}(Z,X)B(\nu,Y)\,,
$$
which is simple to check from the standard definition of covariant derivative of a tensor, and then \eqref{eq:719475giuog} follows.
\end{proof}

{\bf Fermi coordinates.} We introduce Fermi coordinates with respect to \(\Gamma_i\). It will be convenient to follow the notation in \cite{WW19} for them; hence, in the present section, \((y, z)\) with $y\in\R^{n-1}$ and $z\in\R$ will always denote Fermi coordinates with respect to \(\Gamma_i\). Given \(x\in\R^n\) in a neighborhood of \(\Gamma_i\), it is written in these coordinates as \((y, z)\) precisely if:
\begin{itemize}
    \item $y\in\R^{n-1}$ is such that \((y, f_i(y)) \in \Gamma_i\) is the nearest point projection of \(x\). Notice that this is different from the previously defined $x'$, which we recall denotes the projection of any $x\in\R^n$ onto the first $n-1$ coordinates.
    \item \(z\) is the signed distance of \(x\) to \(\Gamma_i\), with positive sign above \(\Gamma_i\).
\end{itemize}
Now, since \(\Gamma_i\) is the graph of \(f_i\), we will use \(y \mapsto (y, f_i(y))\) as a parametrisation of \(\Gamma_i\). The upward unit normal vector of \(\Gamma_i\) at \((y, f_i(y))\) is
$$
N_i(y) := \frac{1}{\sqrt{1 + |\nabla f_i(y)|^2}} \left(-\nabla f_i(y),\, 1 \right).
$$

Perhaps more concisely, but equivalently, the correspondence between $x$ and $(y,z)$ is then just $x=(y, f_i(y))+zN_i(y)$.

Now, simply by \eqref{eq:iashgabg}, i.e. the curvature estimates, these coordinates are well defined in the open set \(\{(y, z) : |z| < c_F |y|,\ |y| > R_0\}\) for some \(c_F > 0\).

For each \(z\), \(\Gamma_i^z\) will denote the smooth hypersurface consisting of all points with signed distance to \(\Gamma_i\) identical to \(z\). By slight abuse of notation, we denote its second fundamental form by $\mathrm{II}_i(z,t)$, and its mean curvature by $H_i(z,t)$. Arguing as in \cite[p. 11-13]{WW19}, we have:
\begin{itemize}
    \item {\bf Mean curvature expansion:}
    $$
        H_i(y, z) = H_i(y,0) + O\left(|z||\mathrm{II}_i|^{2}(y,0)\right) = H_i(y,0) + O\left(|z||y|^{-3}\right).
    $$
    \item {\bf Laplacian in Fermi coordinates:} Denoting by \(\Delta_{i,z}\) the Laplace(-Beltrami) operator over \(\Gamma_i^z\),
    \begin{equation*}
    \Delta = \Delta_{i,z} - H_i(y, z)\partial_z + \partial_{zz}.
    \end{equation*}
    \item {\bf Laplacian over $\Gamma_i^z$:} Let $\varphi$ be a $C^2$ function of $y$ only; then
\begin{equation}\label{eq:47t14gtlvag}
|\Delta_{i,z} \varphi(y) - \Delta_{i,0} \varphi(y)| \leq C|z||\mathrm{II}_i|  \left( |\nabla^2 \varphi(y)| + |\nabla \varphi(y)| \right)\leq C |y|^{-3/2} |z| \left( |\nabla^2 \varphi(y)| + |\nabla \varphi(y)| \right).
\end{equation}
\end{itemize}
{\bf Notation for distances}
\begin{itemize}
    \item Given $x$ with $|x'|> R$, let $d_i$ denote the distance to $\Gamma_{i}$.
    \item We set $\mathcal M_i:=\{|d_i|<|d_{i-1}|\quad\mbox{and}\quad |d_i|<|d_{i+1}|\}$.
\end{itemize}
Moreover,
\begin{itemize}
  \item For $|y| > R$, let $D_i^\pm(y)$ be the distance of $(y, f_i(y))\in\Gamma_i$ to $\Gamma_{i\pm 1}$ respectively.
  \item Denote $D_i(y) := \min\{D_i^+(y), D_i^-(y)\}$.
  \item Set $M(y) := \max_{i} \max_{|\widetilde y| \geq |y|} e^{-\sqrt{2} D_i(\widetilde y)}$.
\end{itemize}

By the bounds on $f_i$ and $\nabla f_i$ in \eqref{eq:endgrowth}, $\Gamma_i$ and $\Gamma_{i+1}$ are rather close in height and almost parallel. By a simple comparison, see for instance \cite[Lemma 8.3]{WW18}, we get:
\begin{lemma}\label{lem:82gwgbogg}
For any $|x| > R_0$,
$$
\left\{
\begin{aligned}
D_i^+(y) &= f_{i+1}(y) - f_i(y) + O(|y|^{-1/6}), \\
D_i^-(y) &= f_i(y) - f_{i-1}(y) + O(|y|^{-1/6}).
\end{aligned}
\right.
$$
\end{lemma}

\noindent {\bf Optimal approximation}

This is taken exactly from \cite{GuiWangWei2020}.\\
Recall the definition of $g$ in \eqref{eq:1dsoldef}; we need a truncated version. Fix a cutoff function $\zeta \in C_c^\infty(-2,2)$, with $\zeta \equiv 1$ in $(-1,1)$ and $|\zeta'| + |\zeta''| \leq 16$.  For all $|y|$ large, let (to ease notation, dependence on $|y|$ will be omitted when convenient)
$$
\bar{g}(t) = \zeta(8 (\log |y|) t) g(t) + [1 - \zeta(8 (\log |y|) t)] \operatorname{sgn}(t), \quad t \in (-\infty, +\infty).
$$
In particular, $\bar{g} \equiv 1$ in $(16 \log |y|, +\infty)$ and $\bar{g} \equiv -1$ in $(-\infty, -16 \log |y|)$.

Then, $\bar{g}$ is an approximate solution to the one-dimensional Allen–Cahn equation, in the sense that
\begin{equation*}
\bar{g}''(t) = W'(\bar{g}(t)) + \bar{\xi}(t)\,,
\end{equation*}
where $\operatorname{spt}(\bar{\xi}) \subset \{8 \log |y| < |t| < 16 \log |y|\}$, and $|\bar{\xi}| + |\bar{\xi}'| + |\bar{\xi}''| \leq C|y|^{-4}$. 

In the following we assume for convenience that $u$ has the same sign as $(-1)^i$ between $\Gamma_i$ and $\Gamma_{i+1}$. Exactly as in \cite[Lemma 4.3]{GuiWangWei2020}, we have:
\begin{lemma}
\label{lem:914gotb2jk}
For any $y\in\R^{n-1}$ with $|y| > R_0$ (perhaps after enlarging $R_0$) and $i\in\{1,...,N\}$, there exists a unique $h_i(y):\R\to\R$ such that in the Fermi coordinates with respect to $\Gamma_i$,
$$
\int_{-\infty}^{+\infty} [u(y, z) - g_*(y, z)] \bar{g}'(z - h_i(y)) \, dz = 0,
$$
where for each $i$, in $\mathcal{M}_i$ we define
$$
g_*(y, z) := g_i + \sum_{j=1}^{i-1} \left[g_j - (-1)^j \right] + \sum_{j=i+1}^N \left[g_j + (-1)^jj \right],
$$
and in the Fermi coordinates $(y, t)$ with respect to $\Gamma_j$,
$$
g_j(y, z) := \bar{g}((-1)^j [z - h_j(y)]).
$$

Moreover, for any $i\in\{1,...,N\}$, we have
$$
\lim_{|y| \to +\infty} \left( |h_i(y)| + |\nabla h_i(y)| + |\nabla^2 h_i(y)| + |\nabla^3 h_i(y)| \right) = 0;
$$
likewise, defining the \textit{approximation error} $\phi := u - g_*$, we have
$$
\lim_{|y| \to +\infty} \sup_{z}\left(|\phi_i(y,z)|+|\nabla \phi_i(y,z)|+|\nabla^2 \phi_i(y,z)|+|\nabla^3 \phi_i(y,z)|\right) = 0;
$$
\end{lemma}
In Fermi coordinates with respect to $\Gamma_i$, the equation for $\phi$ can be written as
\begin{equation*}
\begin{aligned}
\Delta_{i,z}\phi- H_i(y,z)\partial_z\phi + \partial_{zz}\phi &= W''(g_*)\phi+\mathcal{N}(\phi)+ \mathcal{I}+ (-1)^i g_i' \mathcal{R}_{i,1}- g_i'' \mathcal{R}_{i,2} \\
&\quad+\sum_{j \neq i} [(-1)^j g_j'\mathcal{R}_{j,1}-g_j'' \mathcal{R}_{j,2}]-\sum_j \xi_j,
\end{aligned}
\end{equation*}
where
$$
\mathcal{N}(\phi)=W'(g_* + \phi)-W'(g_*)-W''(g_*)\phi=O(\phi^2)
$$
and
$$
\mathcal{I}=W'(g_*)-\sum_j W'(g_j)\,,
$$
while for each $j$, in the Fermi coordinates with respect to $\Gamma_j$,
$$
\xi_j(y, z) = \bar{\xi}((-1)^j(z - h_j(y)))\,,
$$
and
$$
\mathcal{R}_{j,1}(y, z):=H_j(y, z) + \Delta_{j,z} h_j(y)\,, \quad 
\mathcal{R}_{j,2}(y, z):=|\nabla_j h_j(y)|^2\,.
$$

As in \cite[Lemma 4.6]{WW19}, since $u = 0$ on $\Gamma_i$, $h_i$ can be controlled by $\phi$ in the following way.

\begin{lemma}\label{lem:72gioubo}
For each $i$ and $R > R_0$, we have
\begin{equation*}
\|h_i\|_{C^{2,1/2}(\{|y|>R\})} \leq C \left(\|\phi\|_{C^{2,1/2}(\{|y|>R\})} + M(R)\right),
\end{equation*}
as well as the improved tangential derivative bound
\begin{equation*}
\|\nabla_{y} h_i\|_{C^{1,1/2}(\{|y|>R\})} \leq C\left( \|\nabla_{y}\phi\|_{C^{1,1/2}(\{|y|>R\})} + R^{-1/6} M(R)\right).
\end{equation*}
\end{lemma}

\noindent {\bf The Toda system}

By \cite[Section 5]{WW19}, we get the following Toda system:
\begin{equation}\label{eq:todagluing}
H_i + \Delta_{i,0} h_i = \kappa_0 \left(e^{-\sqrt{2} D_i^-} - e^{-\sqrt{2} D_i^+} \right) + Err_i,
\end{equation}
where $\kappa_0$ is a universal constant and $Err_i$ is a higher order error term. More precisely, \cite[Lemma 5.1]{WW19} in our setting becomes:

\begin{lemma}\label{lem:88tgpbgbagu}
For any $|y| > R_0$,
\begin{align*}
|Err_i(y)| &\leq C\Big[ |y|^{-3} + |y|^{-\frac{7}{2}} M(|y| - 100\log |y|) + M(|y| - 100 \log |y|)^{\frac{7}{6}}\\
&\quad + \max_\beta \left( \|H_\beta + \Delta_{\beta,0} h_\beta\|_{C^{1/2}(\{\widetilde y\ :\ |\widetilde y|>|y| - 100 \log |y|\})}^2 + \|\phi\|_{C^{2,1/2}(\{\widetilde y\ :\ |\widetilde y|>|y| - 100 \log |y|\})}^2 \right)\Big].
\end{align*}
\end{lemma}
Naturally, we are using once again that we have the improved bound from \Cref{lem:2ffdec} to obtain the $O(|y|^{-3})$ term.\\

\noindent {\bf Estimates on $\phi$.}

Arguing exactly in the same way as in \cite[Section 6]{WW19}, we have the following.

\begin{lemma}\label{lem:9qbiuvfhs}
There exist two constants $C$ such that for all $R$ large,
\begin{align*}
    \max_i &\left\| H_i + \Delta_{i,0} h_i \right\|_{C^{1/2}(\{|y|>R\})} + \|\phi\|_{C^{2,1/2}(\{ |y|>R\})}\\
    &\leq \frac{1}{2} \Big(\max_i \left\| H_i+ \Delta_{i,0} h_i \right\|_{C^{1/2}(\{|y|>R - 100 \log R\})}\\
    &\qquad+ \|\phi\|_{C^{2,1/2}(\{|y|>R - 100 \log R\})}\Big) 
+ CM(R - 100 \log R) + C R^{-3}.
\end{align*}
\end{lemma}

The constant $1/2$ in front of the first term on the right-hand side of this inequality allows us to iterate for a \textit{finite number} of dyadic scales and conclude that, up to making $R_0$ larger, we have
\begin{equation}\label{eq:90124htp2}
    |H_i(y) + \Delta_{i,0} h_i(y)| + \|\phi\|_{C^{2,1/2}(\{\widetilde y\ :\ |\widetilde y|>|y|\})} \leq C \left[|y|^{-3} + M(|y| - 100 \log |y|)\right] \quad \forall\, |y| \geq R_0.
\end{equation}

The separation error, by the bound in \cite[Proposition 10.1]{WW19}, satisfies $M(|y|) \leq C |y|^{-2}$. Then, applying \eqref{eq:90124htp2} with $R_0$ large enough gives
\begin{equation*}
|H_i(y) + \Delta_{i,0} h_i(y)| + \|\phi\|_{C^{2,1/2}(\{\widetilde y\ :\ |\widetilde y|>|y|\})} \leq C |y|^{-2}.
\end{equation*}

Moreover, we can get a better bound for the horizontal derivatives of $\phi$, by \cite[Proposition 7.1]{WW19}:
\begin{equation}\label{eq:82hiutb2ipb}
\|\nabla_y \phi\|_{C^{1,1/2}(\{\widetilde y\ :\ |\widetilde y|>|y|\})} \leq C |y|^{-2 - 1/7}.
\end{equation}

In view of \Cref{lem:72gioubo}, \eqref{eq:82hiutb2ipb} gives an upgrade for $h$ too:
\begin{equation*}
\|\nabla_{y} h_i\|_{C^{1,1/2}(\{\widetilde y\ :\ |\widetilde y|>|y|\})} \leq C |y|^{-2 - 1/7}.
\end{equation*}
This bounds, in particular, the size of the Laplacian of $h$. Substituting into \eqref{eq:todagluing}, and using \Cref{lem:82gwgbogg} to compare the layer separations, we reach
\begin{equation*}
{\rm div}\left(\frac{\nabla f_{i}}{\sqrt{1+|\nabla f_i|^2}}\right) = \kappa_0 \left(e^{-\sqrt{2}(f_i - f_{i-1})} - e^{-\sqrt{2}(f_{i+1} - f_i)}\right) + O\left(|y|^{-2 - \frac{1}{8}}\right)\,.
\end{equation*}
Since ${\rm div}\left(\frac{\nabla f_{i}}{\sqrt{1+|\nabla f_i|^2}}\right)=\Delta f_{i}+O(|\nabla f_i|^2||D^2f_i|)$, using \eqref{eq:endgrowth} we finally obtain
\begin{equation*}
\Delta f_{i} = \kappa_0 \left(e^{-\sqrt{2}(f_i - f_{i-1})} - e^{-\sqrt{2}(f_{i+1} - f_i)}\right) + O\left(|y|^{-2 - \frac{1}{8}}\right)\,.
\end{equation*}

This shows \eqref{eq:todaimprov}. The stability condition in \eqref{eq:todastabimprov} then follows as in \cite[Proposition 8.1]{WW19}, using the improvements that we have already derived. This concludes the proof of \Cref{prop:todaimprov} and of the present section.

\section{Proofs of the main results}
\subsection{Rigidity for $4 \leq n \leq 7$ --- Proof of Theorems \ref{thm:R4indclass} and \ref{thm:finindclass}}\label{sec:finindclass}

\Cref{thm:R4indclass} will follow directly from \Cref{thm:finindclass}, since the main result in \cite{FSstable} asserts precisely that \eqref{eq:stableclass} holds for $n=4$.

It remains to prove \Cref{thm:finindclass}. Fix $n\in\N$ with $4\leq n \leq 7$, and assume that \eqref{eq:stableclass} is true in $\R^n$. Let $u:\R^n\to[-1,1]$ be a finite index solution to the Allen--Cahn equation, with $\mathcal E(u,B_R)\leq CR^{n-1}$ for some constant $C$ and all $R>0$.

We keep the same setting and notation as in Sections \ref{sec:graphicality}--\ref{sec:todaimprov}. In particular, we have Propositions \ref{prop:endgrowth} and \ref{prop:todaimprov}.

It is easy to conclude if $N=1$ (we will show this at the end of the proof). Assume then for now that there are at least two ends of $u$, i.e. $N\geq 2$---we want to show that this leads to a contradiction.\\
{\bf Farina's integral estimate}

Consider two adjacent ends $\Gamma_i$ and $\Gamma_{i-1}$, $i\in\{2,...,N\}$. Let $v_i := f_i - f_{i-1}$, which depends on $y\in\R^{n-1}$. Applying \eqref{eq:todaimprov} with $i$ and $i-1$, and substracting, we get the inequality
\begin{equation*}
\Delta v_i (y)\leq 2\kappa_0 e^{-\sqrt{2} v_i(y)} + O(|y|^{-2 - 1/8}) \quad \text{in } \R^{n-1}\setminus \overline{B_{R_0}}.
\end{equation*}
In terms of $V_i:= e^{-\sqrt{2}v_i}$, this gives
\begin{equation*}
- \Delta V_i \leq 2\sqrt{2}\kappa_0 V_i^2- V_i^{-1} |\nabla V_i|^2 + O(|y|^{-2 - 1/8}) V_i \quad \text{in } \R^{n-1}\setminus \overline{B_{R_0}}.
\end{equation*}
We argue as in \cite{Far07} (see also \cite{DF09,Wang19,GuiWangWei2020}) and obtain an integral estimate for the interactions:

Let $q > 0$ and $\eta \in C_0^\infty (\R^{n-1}\setminus \overline{B_{R_0}})$. Multiplying the equation for $V_i$ by $V_i^{2q - 1} \eta^2$ and integrating by parts we find
\begin{align*}
(2q-1) \int_{\R^{n-1}\setminus \overline{B_{R_0}}} V_i^{2q - 2} &|\nabla V_i|^2 \eta^2  dy+2\int_{\R^{n-1}\setminus \overline{B_{R_0}}} V_i^{2q - 1} \eta \nabla V_i \cdot \nabla \eta(y)  dy \leq 2\sqrt{2}\kappa_0 \int_{\R^{n-1}\setminus \overline{B_{R_0}}} V_i^{2q + 1} \eta^2  dy  \\
&-\int_{\R^{n-1}\setminus \overline{B_{R_0}}} V_i^{2q - 2} |\nabla V_i|^2 \eta^2  dy + C \int_{\R^{n-1}\setminus \overline{B_{R_0}}} V_i^{2q} \eta^2 |y|^{-2 - 1/8} dy.
\end{align*}
Integrating by parts in the second integral, and collecting terms,
\begin{align}
2q \int_{\R^{n-1}\setminus \overline{B_{R_0}}} V_i^{2q - 2} |\nabla V_i|^2 \eta^2  dy &\leq 2\sqrt{2}\kappa_0 \int_{\R^{n-1}\setminus \overline{B_{R_0}}} V_i^{2q + 1} \eta^2  dy \notag \\
&\quad + C \int_{\R^{n-1}\setminus \overline{B_{R_0}}} V_i^{2q} \left[ |\nabla \eta|^2 + \eta |D^2 \eta| + \eta^2 |y|^{-2 - 1/8} \right] dy.\label{eq:128fgouiqbgf}
\end{align}

On the other hand, substituting $V_i^q \eta$ as a test function into \eqref{eq:todastabimprov}:
\begin{align}
2\sqrt{2}\kappa_0 \int_{\R^{n-1}\setminus \overline{B_{R_0}}} V_i^{2q + 1} \eta^2 dy
&\leq q^2 \left(1 + C R_0^{-\frac{1}{8}} \right) \int_{\R^{n-1}\setminus \overline{B_{R_0}}} V_i^{2q - 2} |\nabla V_i|^2 \eta^2 dy \notag \\
&\quad + C \int_{\R^{n-1}\setminus \overline{B_{R_0}}} V_i^{2q} \left[ |\nabla \eta|^2 + \eta^2 |y|^{-2 - 1/8} \right] dy. \label{eq:128fgouiqbgf2}
\end{align}

Combining \eqref{eq:128fgouiqbgf} and \eqref{eq:128fgouiqbgf2} we find
\begin{align*}
\left[2q-q^2 \left(1 + C R_0^{-\frac{1}{8}} \right)\right] \int_{\R^{n-1}\setminus \overline{B_{R_0}}} V_i^{2q - 2} &|\nabla V_i|^2 \eta^2  dy\\
&\leq C \int_{\R^{n-1}\setminus \overline{B_{R_0}}} V_i^{2q} \left[ |\nabla \eta|^2 + \eta |D^2 \eta| + \eta^2 |y|^{-2 - 1/8} \right] dy\,.
\end{align*}
If $q < 2$, up to making $R_0$ large enough we find some $C(q) < +\infty$ such that
\begin{align*}
\int_{\R^{n-1}\setminus \overline{B_{R_0}}} V_i^{2q - 2} &|\nabla V_i|^2 \eta^2  dy\\
&\leq C(q) \int_{\R^{n-1}\setminus \overline{B_{R_0}}} V_i^{2q} \left[ |\nabla \eta|^2 + \eta |D^2 \eta| + \eta^2 |y|^{-2 - 1/8} \right] dy\,.
\end{align*}
Combining this back with \eqref{eq:128fgouiqbgf2},
\begin{equation}\label{eq:81tgqbgpbpg}
\int_{\R^{n-1}\setminus \overline{B_{R_0}}} V_i^{2q + 1} \eta^2  dy \leq C(q) \int_{\R^{n-1}\setminus \overline{B_{R_0}}} V_i^{2q} \left[ |\nabla \eta|^2 + \eta |D^2 \eta| + \eta^2 |y|^{-2 - 1/8} \right]  dy.
\end{equation}

Assume $0 \leq \eta \leq 1$. We can then replace $\eta$ by $\eta^m$, for some $m \geq 5$; applying Hölder's inequality to \eqref{eq:81tgqbgpbpg}, we get
\begin{equation}\label{eq:32gtwbot2g}
\int_{\R^{n-1}\setminus \overline{B_{R_0}}} V_i^{2q + 1} \eta^2 dy \leq C(q) \left( \int_{\R^{n-1}\setminus \overline{B_{R_0}}} \left[|\nabla \eta|^2 + |D^2 \eta| + |y|^{-2 - 1/8} \right]^{2q + 1} dy \right)^{\frac{1}{2q + 1}}.
\end{equation}

For any $R > 2R_0$, take $\eta_R \in C_c^\infty(R_0, 2R)$ such that $0 \leq \eta_R \leq 1$, $\eta_R \equiv 1$ in $(2R_0, R)$, and $|\nabla \eta_R|^2 + |D^2 \eta_R| \leq C R^{-2}$ in $(R_0, 2R)$. Substituting $\eta_R(|x|)$ into \eqref{eq:32gtwbot2g}, we get
\begin{equation*}
\int_{\R^{n-1}\setminus \overline{B_{2R_0}}} V_i^{2q + 1} dy \leq C + C R^{(n-1) - 2(2q + 1)}.
\end{equation*}

Thanks to our assumptions on $n$, we can take $q\in(0,2)$ such that $2q + 1 = (n-1)/2$. After letting $R \to +\infty$ in the inequality, we arrive at
\begin{equation*}
\int_{\R^{n-1}\setminus \overline{B_{2R_0}}} V_i^{(n-1)/2} dy \leq C.
\end{equation*}
We emphasise that this integrability estimate breaks the natural scaling $V_i=O(|x|^{-2})$ for the Toda system.

{\bf Contradiction from $N\geq 2$.}

From the above estimate, we deduce in particular that given $\delta>0$, there is some $R_\delta$ such that
$$\int_{\R^{n-1}\setminus \overline{B_{R_\delta}}} V_i^{n/2} dy\leq \delta\,.$$
Then, as in Dancer--Farina \cite{DF09}, a Harnack-type inequality (from \cite{Serrin64, Tru67}, and applied to the differential inequality satisfied by $V_i$) readily implies that $|y|^2V_i(y)\leq C\delta$ for all $|y|$ large enough. Since $\delta>0$ is arbitrary, this shows that
$$
e^{-\sqrt{2} v_i(y)} = o(|y|^{-2})\quad\mbox{as}\quad y\to\infty\,.
$$

This then leads to a contradiction as in \cite{DF09}:
\begin{itemize}
\item The bound $
e^{-\sqrt{2} v_i(y)} = o(|y|^{-2})$ implies---since $\Delta v_i\leq e^{-\sqrt{2} v_i(y)}+O(|y|^{-2-1/8})$---that $
\Delta v_i \leq o(|y|^{-2})$ too.
\item This shows then---for the radially averaged function $\bar v_i$---by integrating along rays that
$$\bar v_i\leq \delta\log|y|+C_\delta(1+|y|^{-{n-3}})\,,$$
for any $\delta>0$.
\item But then $
e^{-\sqrt{2} \bar v_i(y)}\geq c_\delta|y|^{-\delta}$, which choosing $\delta<2$ clearly gives a contradiction with
$$
e^{-\sqrt{2} \bar v_i(y)}\leq\max_{\{z:|z|=|y|\}}e^{-\sqrt{2} v_i(z)} \leq C|y|^{-2}\,.$$
\end{itemize}

{\bf Classification with $N=1$.}

The previous arguments confirm that $u$ only has one end indeed, i.e. $N=1$. The one-dimensional symmetry of $u$ is then obtained for example as follows: by the discussion in \Cref{rmk:inftydens} we deduce then that ${\bf M}_\infty(u)=1$. But then, Wang's Allard-type result (recorded in \Cref{thm:ACallard}) shows that $u$ is one-dimensional as desired.

\subsection{Behaviour in closed 4-manifolds --- Proof of \Cref{thm:4closedbeh}}\label{sec:4closedbeh}
We recall first that there is a local monotonicity formula (with a controlled error) for solutions to A--C in the Riemannian setting, see \cite[Appendix B]{Guaraco2018}. In particular, it shows that there is some $\Lambda$---depending only on $M$ and the energy bound $E_0$---such that ${\bf M}_r(u_{\ep_i},p):=\frac{1}{r^{n-1}}\int_{B_r(p)} \frac{1}{\sigma_{n-1}}\left[\ep|\nabla u_{\ep_i}|^2+\frac{1}{\ep}W(u_{\ep_i})\right]\leq \Lambda$ for every $p\in M$ and $r>0$. This will allow us to perform rescaling arguments while ensuring a bound on the energy densities at all scales of interest.

\noindent {\bf Step 1.} Lower bound for the gradient .\\
Assume, for contradiction, that there were a subsequence $\ep_i\to 0$ and some $p_i\in M$ such that 
    \begin{equation}\label{eq:fiowgopgagop2}
         |u_{\ep_i}|(p_i)\leq 0.9 \quad\mbox{and}\quad|\nabla u_{\ep_i}|(p_i)<\frac{1}{i\ep_i}\,.
    \end{equation}
    Let $g$ denote the metric of $M$, and consider the rescaled manifolds $\widetilde M_i:=(M,\widetilde g_i)$, where $\widetilde g_i:=\ep_i^{-1/2}g$. In Riemannian normal coordinates, centered at $p_i$ (or any sequence of points, for that matter), the metrics of the $\widetilde M_i$ are converging---in $C^k_{loc}$, for any $k\in\N$---to the Euclidean metric on all of $\R^n$.
    
    Let $\widetilde u_i$ denote the functions $u_{\ep_i}$, viewed over $\widetilde M_i$; they are now solutions to A--C with parameter $1$. Moreover, ${\bf M}_r(\widetilde u_i,p_i)\leq \Lambda$ as well, again interpreted over $\widetilde M_i$ now. Additionally, by \eqref{eq:fiowgopgagop2} we have $|\widetilde u_i|(p_i)\leq 0.9$ and $|\nabla \widetilde u_i|(p_i)<\frac{1}{i}$.
    
    Now, since $|\widetilde u_i|\leq 1$, and they all satisfy \eqref{eq:aceqintro}, we obtain uniform $C^3$ bounds for the $\widetilde u_i$ (by standard $C^3$ bootstrap estimates for semilinear elliptic equations). Then, up to passing to Riemannian normal coordinates around $p_i$, the Arzel\`a-Ascoli theorem provides a subsequence converging in $C^2_{\rm loc}(\R^n)$ to $\widetilde u_\infty$, a bounded solution to A--C on all of $\R^n$, such that
    $${\bf M}_\infty(\widetilde u_\infty)=\lim_{r\to\infty}\lim_{i\to\infty} {\bf M}_r(\widetilde u_i, p_i)\leq \Lambda\qquad\mbox{yet}\qquad |\widetilde u_\infty|(0)\leq 0.9\quad\mbox{and}\quad|\nabla \widetilde u_\infty|(0)=0\,.$$
    Finally, it is easy to see that $\widetilde u_\infty$ has finite Morse index, in fact bounded by $C_0$ just as the $\widetilde u_i$. Then, by \Cref{thm:R3indclass}---together with the fact that $|\widetilde u_\infty|(0)\leq 0.9<1$---we deduce that $\widetilde u_\infty=g(a\cdot x+b)$ for some appropriate $a,b$. But then $|\nabla \widetilde u_\infty|(0)=g'(b)\neq 0$, which gives a contradiction.
    
    \noindent \textbf{Step 2.} Proof of the bound for $\mathcal A_{u_{\ep}}$.

    Since we want to prove an asymptotic in $\ep_i$ small, we can assume that $\ep_i\to 0$. Now, by Step 1, we know then in particular that $\{u_{\ep_i}=t\}$ is smooth as long as $|t|\leq 0.9$ and $i$ is large enough.
    Assume, by contradiction, that there were a subsequence (not relabeled) of solutions for which the asymptotic fails. This means that for some $\delta>0$ there are $p_i\in \{|\widetilde u_i|\leq 0.9\}$ with $\ep_i\mathcal A_{u_{\ep_i}}(p_i)\to\infty$.
    
    We consider then again the rescaled manifolds $\widetilde M_i:=(M,\ep_i^{-1/2}g)$ from Step 1, with associated $\widetilde u_i$ satisfying $\mathcal A_{\widetilde u_i}(p_i)\geq \delta$. Passing to a limit---in Riemannian normal coordinates centered at $p_i$---via Arzel\`a--Ascoli, we obtain a solution $\widetilde u_\infty$ to A--C in all of $\R^n$. Moreover, it has bounded index and energy density since the $\widetilde u_i$ do. As in Step 1, then $\widetilde u_\infty=g(a\cdot x+b)$ for some appropriate $a,b$. On the other hand, it necessarily satisfies $\mathcal A_{\widetilde u_\infty}(0)\geq \delta$, which since $\mathcal A_g\equiv 0$ yields a contradiction.

    \noindent \textbf{Step 3.} Proof of the sheet separation.\\
    Given $\Lambda\geq 1$, by the previous steps there is $\ep_0>0$ such that given $\ep\leq \ep_0$ and $t\in[-0.9,0.9]$, the following holds: Given $p\in\{u_\ep=t\}$, the level set $\{u_\ep=t\}$ decomposes into several parallel graphs in normal coordinates in $B_{2\Lambda\ep}(p)$.
    
    Arguing via a rescaling argument once again, just as in the previous steps, we then deduce that---up to making $\ep_0$ small enough---in fact only a single graph can pass through $B_{\Lambda\ep}(p)$, since $\{g(a\cdot x+b)=t\}\subset \R^n$ consists of a single hyperplane.

\subsection{Structure for $n=3$ --- Proof of \Cref{thm:R3indclass}}\label{sec:R3indclass}
We keep once again the setting and notation from Sections \ref{sec:todaimprov} and \ref{sec:finindclass}, restricting now to $n=3$.

If $N=1$, then $u$ is actually one-dimensional, just as in the end of the proof of \Cref{thm:finindclass}. Therefore, we assume $N\geq 2$.

{\bf Integrability of layer interaction}

Let $R>4R_0$. We start by testing stability (i.e. \eqref{eq:todastabimprov}) with a standard linear cutoff $\eta_R\in C_c^1(B_{2R}\setminus \overline{B_{R_0}})$, satisfying $\eta_R\equiv 1$ in $B_{R}\setminus \overline{B_{2R_0}}$ and $|\nabla \eta_R|\leq C R^{-1}$. This shows that, for any $i\in\{2,...,N\}$,
$$
\int_{B_R'\setminus \overline{B_{2R_0}'}} e^{-\sqrt{2}(f_i-f_{i-1})}\,dy\leq C\int_{B_{2R}'\setminus \overline{B_{R_0}'}}\frac{1}{R^2}\,dy +C\int_{B_{2R}'\setminus \overline{B_{R_0}'}} |y|^{-2-\frac{1}{8}}\,dy\leq C\,,
$$
where $C$ is independent from $R$. Sending $R\to\infty$, we conclude then that 
\begin{equation}\label{eq:238hrfbqvap9}
    \int_{\R^2\setminus \overline{B_{2R_0}'}} e^{-\sqrt{2}(f_i-f_{i-1})}\,dx'\leq C\quad\mbox{for every}\quad i=2,...,N\,.
\end{equation}
Fix $i\in\{1,...,N\}$ now. Integrating in \eqref{eq:todaimprov} with this fixed value of $i$, and using \eqref{eq:238hrfbqvap9} and the integrability of $|y|^{-2-\frac{1}{8}}$ in $\R^2$, we conclude that
\begin{equation}\label{eq:245gtqfj}
    \Delta f_i=h_i\qquad\mbox{for some}\ h_i\in (L^1\cap L^\infty)(\R^{n-1}\setminus \overline{B_{2 R_0}'})\,.
\end{equation}
It is hard to obtain any more information from stability. Nevertheless, the preceding information about the $f_i$ will already strongly constrain their structure.

{\bf Structure of single layers}

The first idea, which is inspired by \cite{CL91}, is to split $f_i=v_i^a+v_i^b$ with 
    $$v_i^a(y):=\frac{1}{2\pi}\int_{\R^2\setminus \overline{B_{2R_0}'}} (\log|y-z|-\log |z|)| h(z)\,dz$$
    and $v_i^b:=f_i-v_i^a$.

    Observe that $\int \log |z|\, h(z)\,dz$ is ``just a renormalisation constant''. On the other hand, it could very well be infinite---the integral defining $v_i^a$ is only seen to be convergent (using the integrability assumptions for $h$, and the fact that $\left|\log|y-z|-\log|z|\right|\leq C\frac{|y|}{|z|}\leq C$ for $|z|\geq 2|y|$) upon considering its integrand as a whole.
    
    It is simple to see that, setting $c_i^a:= \frac{1}{2\pi}\int_{\R^2\setminus \overline{B_{2R_0}'}} h\in\R$, we have
    $$v_i^a(y)=(c_i^a+o(1))\log |y| \quad\mbox{as}\quad |y|\to\infty \,.$$
    This is still far from the type of asymptotics desired in \Cref{sec:R3indclass}. However, at this point we can obtain more precise asymptotics for $v_i^b$. Observe that $v_i^b=f-v_i^a$ satisfies $\Delta v_i^b=0$, i.e. it is harmonic, thus it is not limited by the Laplacian bound anymore. A priori, $v_i^b$ could be any harmonic function; on the other hand, by the growth bound for $v$ and the asymptotics for $v_i^a$, we additionally know that $|v_i^b(y)|\leq C|y|^{1/2}$.
    
    Via the improvement of flatness in annuli strategy, we will show:
\begin{proposition}\label{prop:logexp3D}
    Let $w:\R^2\setminus \overline{B_{2 R_0}'}\to\R$ be a harmonic function satisfying $|w(y)|\leq C|y|^{1/2}$.
    
    Then, there are $b,c\in\R$ such that $w=b+c\log|y|+O(|y|^{-\alpha})$ for every $\alpha\in (0,1)$.
\end{proposition}
We will prove \Cref{prop:logexp3D} at the end of the section. For now, we show how to conclude from this:\\

First, applying \Cref{prop:logexp3D} to $w=v_i^b$, we deduce the existence of $b_i^b,c_i^b\in\R$ such that
\begin{equation}\label{eq:2543h02th20}
    v_i^b(y)=b_i^b+c_i^b\log|y|+O(|y|^{-\alpha})\,.
\end{equation}

It remains to improve the asymptotics for $v_i^a$. Let $c_i:=c_i^a+c_i^b$; we want to show that $c_{i+1}-c_i> \sqrt{2}$.

{\bf Separation bounds}

Since $e^{-\sqrt{2}(v_{i+1}-v_i)}$ is integrable, and
$$e^{-\sqrt{2}(v_{i+1}-v_i)}=e^{-\sqrt{2}(c_{i+1}-c_i+o(1))\log|y|}\,,$$
we directly see that $c_{i+1}-c_i\geq \sqrt{2}$.\\

Assume for contradiction that the strict bound were false for at least one pair of consecutive layers, and let $i_1=\min\{i:c_{i+1}-c_{i}=\sqrt{2}\}$ and $i_2=\min\{i>i_1:c_{i+1}-c_{i}>\sqrt{2}\}$. These layers satisfy then
\begin{equation}\label{eq:174ggqbpg}
    e^{-(v_{i_1}-v_{i_1-1})}=O(|y|^{-2-\alpha})\quad\mbox{and}\quad e^{-(v_{i_2+1}-v_{i_2})}=O(|y|^{-2-\alpha})\quad\mbox{for some}\quad \alpha>0\,,
\end{equation}
and $c_{i+1}-c_i=\sqrt{2}$ for every $i\in\{i_1,...,i_2-1\}$.\\

We claim we have the one-sided bounds
$$v_{i_1}(y)\geq c_{i_1}\log|y|-C \quad\mbox{and}\quad v_{i_2}(y)\leq c_{i_2}\log|y|+C\,.$$
Indeed, observe that---by \eqref{eq:todaimprov} and \eqref{eq:174ggqbpg}---the top layer $f_{i_2}$ satisfies $\Delta f_{i_2}=\kappa_0e^{-\sqrt{2}(f_{i_2}-f_{i_2-1})}+O(|y|^{-2-\alpha})$, which has positive leading term on the right. Write
\begin{align*}
    v_{i_2}^a(y)&=\frac{1}{2\pi}\int_{\R^2\setminus \overline{B_{2R_0}'}} (\log|y-z|-\log |z|) \Delta v_{i_2}\,dz\\
    &= \frac{1}{2\pi}\log|y|\int_{\R^2\setminus \overline{B_{2R_0}'}} \Delta v_{i_2}\,dz+ \frac{1}{2\pi}\int_{\R^2\setminus \overline{B_{2R_0}'}} \log\frac{|y-z|}{|y||z|}\ \Delta v_{i_2}\,dz\\
    &= c_{i_2}^a\log|y|+ \frac{1}{2\pi}\int_{\R^2\setminus \overline{B_{2R_0}'}} \log\frac{|y-z|}{|y||z|}\ \left[\kappa_0 e^{-\sqrt{2}(f_{i_2}-f_{i_2-1})}+O(|y|^{-2-\alpha})\right]\,dz\,dz\\
    &\leq c_{i_2}^a\log|y|+ \frac{1}{2\pi}\int_{\R^2\setminus \overline{B_{2R_0}'}} \log\frac{|y-z|}{|y||z|}\ \kappa_0e^{-\sqrt{2}(f_{i_2}-f_{i_2-1})}\,dz+C\,.
\end{align*}
We can assume that $R_0\geq 2$, so that we only need to consider $|z|,|y|\geq 2$. But then, arguing as in \cite[Lemma 2.1]{Lin98}, we can bound
$$\frac{|y-z|}{|y||z|}\leq \frac{|y|+|z|}{|y||z|}\leq \frac{|y||z|}{|y||z|}=1\,,$$
and therefore $\log \frac{|y-z|}{|y||z|}\leq 0$, which shows that
\begin{align*}
    v_{i_2}^a(y)\leq c_{i_2}^a\log|y|+ C\,,\quad\mbox{or}\quad v_{i_2}\leq c_{i_2}\log|y|+C
\end{align*}
as desired. The lower bound for $v_{i_1}$ follows similarly.\\

We proceed with the contradiction argument:

Define $Err_i(y)$ via $v_i(y)=c_i\log|y|+Err_i(y)$. The one-sided bounds show that
\begin{equation}\label{eq:7231ftpqg}
    Err_{i_2}\leq C\quad\mbox{and}\quad Err_{i_1}\geq -C\,.
\end{equation}
On the other hand, by \eqref{eq:238hrfbqvap9} we see that, for any $i\in\{i_1,...,i_2-1\}$,
\begin{align*}
    +\infty>\int_{\R^2\setminus \overline{B_{R_0}}}e^{-\sqrt{2}(v_{i+1}-v_{i})}&=\int_{\R^2\setminus \overline{B_{R_0}}}e^{-\sqrt{2}\left[(c_{i+1}-c_{i})\log|y|+(Err_{i+1}-Err_{i})\right]}\\
    &=\int_{\R^2\setminus \overline{B_{R_0}}}\frac{1}{|y|^2}e^{-\sqrt{2}(Err_{i+1}-Err_{i})}\,.
\end{align*}
Define the measure $\mu(A):=\int_{A}\frac{1}{|y|^2}$, $A\subset \R^2\setminus \overline{B_{R_0}}$. Very much in particular, the above shows that, for any $M>0$, we have
\begin{align*}
    \mu(\{y:(Err_{i+1}-Err_{i})(y)\leq M\})&=\mu(x:\{e^{-\sqrt{2}(Err_{i+1}-Err_{i})(y)}\geq  e^{-\sqrt{2}M}\})\\
    &\leq e^{\sqrt{2}M}\int_{\R^2\setminus \overline{B_{R_0}}}\frac{1}{|y|^2}e^{-\sqrt{2}(Err_{i+1}-Err_{i})}\\
    &<\infty\,.
\end{align*}
By a union bound then
$$\mu\left(\left\{y:(Err_{i+1}-Err_{i})(y)\leq M\quad \mbox{for some}\quad i\in\{i_1,...,i_2-1\}\right\}\right)<\infty\quad\mbox{as well}\,;$$
since $\frac{1}{|y|^2}$ is not integrable in $\R^2\setminus \overline{B_{R_0}}$, or equivalently $\mu(\R^2\setminus \overline{B_{R_0}})=\infty$,
$$\mbox{given}\quad M>0\quad\mbox{there exists}\quad y\in \R^2\setminus \overline{B_{R_0}}\quad\mbox{with}\quad (Err_{i+1}-Err_{i})(y)> M \quad \forall i\in\{i_1,...,i_2-1\}\,.$$
This means, in particular, that for this $y$ we have
$$(Err_{i_2}-Err_{i_1})(y)=\sum_{i=i_1}^{i_2-1}(Err_{i+1}-Err_{i})(y)> M\,.$$
But then, since (by \eqref{eq:7231ftpqg}) we had the uniform bound $Err_{i_2}-Err_{i_1}\leq C$, the fact that $M>0$ can be taken arbitrarily large yields a contradiction.

{\bf Improvement of $v_i^a$}

Now that we know that $c_{i+1}-c_i> \sqrt{2}$ for every $i\in\{1,...,N-1\}$, by the asymptotics for the $v_i$ we conclude that $e^{-\sqrt{2}(v_{i+1}-v_i)}=O(|y|^{-2-\alpha})$ for some $\alpha>0$. Then, from \eqref{eq:todaimprov} and the integral defining $v_i^a$ we easily deduce that
$$v_i^a=b_i^a+c_i^a\log|x|+O(|y|^{-\alpha})\,,$$
which---letting $b_i:=b_i^a+b_i^b$---together with the claimed asymptotics for the $v_i^b$ shows \Cref{thm:R3indclass}.

{\bf A catenoidal improvement of flatness}

It remains to prove \Cref{prop:logexp3D}. The main idea is to run an enhanced version of the improvement of flatness in annuli iteration from \Cref{sec:annimprovhyper}, which allows for the presence of logarithms (and constant functions) at each scale.

Our full decay statement takes the following form in this case:
\begin{proposition}[Sublinearity implies decay]\label{prop:dec2B}
Let $v:\R^2\setminus \overline{B_{2 R_0}'}\to\R$ be harmonic. Let $\ep>0$, and assume that
\begin{align*}
    |v(x')|\leq \ep |x'|\quad\mbox{in}\quad \R^2\setminus \overline{B_{2R_0}'}\,.
\end{align*}
Given $\alpha\in(0,1)$ there is $C=C(\alpha,R_0)>0$, as well as some $\widetilde a_k\in\R^2$ and $\widetilde b_k, \widetilde c_k\in\R$ for every $k\in\N$ with $2^k\geq 4R_0$, such that
\begin{align*}
    |v(x')-\widetilde a_k\cdot x'-\widetilde b_k-\widetilde c_k \log|x'||\leq C\ep 2^{ -\alpha k}\quad\mbox{in}\quad B_{2^{k+1}}\setminus \overline{B_{2^{k-1}}}\,.
\end{align*}
\end{proposition}
As in \Cref{sec:annimprovhyper}, we will reduce this result to an ``improvement of flatness''-type iteration.

\begin{proposition}[Main iteration]\label{prop:mainiterationB}
    Given $\alpha\in(0,1)$, there are $k_0,K\in\N$ depending on $n$ and $\alpha$ such that the following holds. Set $\delta=2^{-\alpha-1}$. Assume we are given:
    \begin{itemize}
        \item Some base scale $k_1\in\N$, with $k_1\geq K$.
        \item Some ``best flatness'' $\ep>0$, with $2^{(-\alpha-1) k_1}\leq \ep$.
        \item Some $a_k\in \R^2$ and $b,c\in\R$ for all $k\geq k_1-k_0$, such that the following hold:
\begin{itemize}
    \item Excess decay between $k_1-k_0$ and $k_1$:
    $$|v(x')-a_k\cdot x'-b_k-c_k\log|x'||\leq \delta^{k-k_1}\ep 2^{k} \mbox{ for all } x'\in B_{2^{k+1}}\setminus \overline{B_{2^{k-1}}},\,\quad k_1-k_0\leq k \leq k_1\,.$$
    \item Preservation of best flatness: $$|v(x')-a_k\cdot x'-b_k-c_k\log|x'||\leq \ep 2^{k} \mbox{ for all } x'\in B_{2^{k+1}}\setminus \overline{B_{2^{k-1}}},\,\quad k_1+1\leq k \,.$$
\end{itemize}
    \end{itemize}
Then, we have the following:
    \begin{itemize}
        \item Set the new base scale $k_1':=k_1+1\in\N$, which satisfies $k_1'\geq K$.
        \item Set the new ``best flatness'' $\ep'=\delta\ep>0$, which satisfies $2^{(\alpha-1) k_1'}\leq \ep'$.
        \item Then, there are some new $a_k'\in \R^2$ and $b_k',c_k'\in\R$ for all $k\geq k_1'-k_0$, such that the following hold:
\begin{itemize}
    \item Excess decay between $k_1'-k_0$ and $k_1'$:
    $$|v(x')-a_k'\cdot x'-b_k'-c_k'\log|x'||\leq \delta^{k-k_1}\ep' 2^{k} \mbox{ for all } x'\in B_{2^{k+1}}\setminus \overline{B_{2^{k-1}}},\,\quad k_1'-k_0\leq k \leq k_1'\,.$$
    \item Preservation of best flatness: $$|v(x')-a_k'\cdot x'-b_k'-c_k'\log|x'||\leq \ep' 2^{k} \mbox{ for all } x'\in B_{2^{k+1}}\setminus \overline{B_{2^{k-1}}},\,\quad k_1'+1\leq k \,.$$
\end{itemize}
    \end{itemize}
\end{proposition}
We first show how \Cref{prop:dec2B} is implied by \Cref{prop:mainiterationB}, and how \Cref{prop:logexp3D} is implied by \Cref{prop:dec2B}.
\begin{proof}[Proof of \Cref{prop:dec2B}, assuming \Cref{prop:mainiterationB}]
    Fix $\alpha\in (0,1)$. Let $k_0,K\in\N$ be given by \Cref{prop:mainiterationB}, and set $\delta=2^{-\alpha-1}$. By assumption (in \Cref{prop:dec2B}), we have
    \begin{equation}\label{eq:67ehjkbgavbB}
        |v(x')|\leq \ep 2^k \quad\mbox{in}\quad (B_{2^{k+1}}\setminus \overline{B_{2^{k-1}}})\quad\mbox{for every}\quad k\geq \lceil \log_2{R_0}\rceil\,.
    \end{equation}
    It is clear now how to proceed:
    \begin{itemize}
        \item Set $k_1:=\max\{K,k_0+\lceil \log_2{R_0}\rceil\}$, so that $k_1\geq K$.
        \item Up to making $k_1$ larger, we also have $2^{(\alpha-1)k_1}\leq \ep$.
        \item Just set $a_k,b_k,c_k$ equal to $0$. Since (by definition) $k_1-k_0\geq \lceil \log_2{R_0}\rceil$, by \eqref{eq:67ehjkbgavbB} it is immediate to see that the excess decay for $k_1-k_0\leq k \leq k_1$ and preservation of best flatness for $k_1+1\leq k$ hold (with best flatness $\ep$).
    \end{itemize}
    This means that the hypotheses of \Cref{prop:mainiterationB} are satisfied. Iterating the proposition, setting $k_{1,l}:=k_1+l$ and $\ep_l:=\delta^l \ep$, $l\in\N$, the ``preservation of best flatness'' part gives---at the $l$-th iteration---coefficients $\widetilde a_{k_{1,l}}\in\R^2$ and $\widetilde b_{k_{1,l}},\widetilde c_{k_{1,l}}\in\R$ such that
    $$|v(x')-\widetilde a_{k_{1,l}}\cdot x'-\widetilde b_{k_{1,l}}-\widetilde c_{k_{1,l}}\log |x'||\leq \ep_l2^{k_{1,l}}= \delta^l \ep 2^{k_{1,l}}\quad\mbox{in}\quad  B_{2^{k_{1,l}+1}}\setminus \overline{B_{2^{k_{1,l}-1}}}\,.$$
    Letting $C:=\delta^{-k_1}$, which does not depend on $l$, we can rewrite this as
    $$|v(x')-\widetilde a_{k_{1,l}}\cdot x'-\widetilde b_{k_{1,l}}-\widetilde c_{k_{1,l}}\log |x'||\leq  C\ep\delta^{k_{1,l}} 2^{k_{1,l}}\quad\mbox{in}\quad B_{2^{k_{1,l}+1}}\setminus \overline{B_{2^{k_{1,l}-1}}}\,,$$
    or (since $\delta=2^{-\alpha-1}$)
    $$|v(x')-\widetilde a_{k_{1,l}}\cdot x'-\widetilde b_{k_{1,l}}-\widetilde c_{k_{1,l}}\log |x'||\leq  C\ep 2^{-\alpha k_{1,l}}\quad\mbox{in}\quad  B_{2^{k_{1,l}+1}}\setminus \overline{B_{2^{k_{1,l}-1}}}\,,$$
    which clearly proves \Cref{prop:dec2B}.
\end{proof}
\begin{proof}[Proof of \Cref{prop:logexp3D}, assuming \Cref{prop:dec2B}]
    The condition $|v(x')|\leq C|x'|^{1/2}$ obviously implies that there is $R_0$ such that $|v(x')|\leq |x'|$ in $\R^2\setminus \overline{B_{R_0}}$.
    
    Then, \Cref{prop:dec2B} gives $\widetilde a_{k}\in\R^2$ and $\widetilde b_{k},\widetilde c_{k}\in\R$ for every $k\in\N$ such that
\begin{align}\label{eq:37y4tgweeubB}
    |v(x')-\widetilde a_{k}\cdot x'-\widetilde b_k-\widetilde c_k\log |x'||\leq  C 2^{- \alpha k}\quad\mbox{in}\quad B_{2^{k+1}}\setminus \overline{B_{2^{k-1}}}\,.
\end{align}
\noindent {\bf Step 1.} Coefficient comparison.

We will show that
\begin{equation}\label{eq:841tguqB}
    |a_{k+1}-a_k|\leq C 2^{(-\alpha-1)k}\log 2^k\,,\quad |b_{k+1}-b_k|\leq C 2^{- \alpha k}\log 2^k\,,\quad |c_{k+1}-c_k|\leq C 2^{- \alpha k}\,.
\end{equation}
To see this, consider two consecutive scales in \eqref{eq:37y4tgweeubB}. The triangle inequality gives then
\begin{align}\label{eq:781236r098froiuavgB}
    |(\widetilde a_{k+1}-\widetilde a_{k})\cdot x'+(\widetilde b_{k+1}-\widetilde b_k)+(\widetilde c_{k+1}-\widetilde c_k)\log |x'||\leq  C 2^{- \alpha k}\quad\mbox{in}\quad B_{2^{k}}\setminus \overline{B_{2^{k-1}}}\,.
\end{align}
In particular, we find that
\begin{align*}
    |(\widetilde b_{k+1}-\widetilde b_k)+(\widetilde c_{k+1}-\widetilde c_k)\log |x'||\leq  C 2^{- \alpha k}\quad\mbox{in}\quad \{(\widetilde a_{k+1}-\widetilde a_{k})\cdot x'=0\}\cap (B_{2^{k}}\setminus \overline{B_{2^{k-1}}}) \,,
\end{align*}
thus (testing with $x'$ in the inner and outer boundaries of the domain) we find
\begin{align*}
    |(\widetilde b_{k+1}-\widetilde b_k)+(\widetilde c_{k+1}-\widetilde c_k)\log 2^k|\leq  C 2^{- \alpha k} \quad\mbox{and}\quad |(\widetilde b_{k+1}-\widetilde b_k)+(\widetilde c_{k+1}-\widetilde c_k)\log 2^{k-1}|\leq  C 2^{- \alpha k}\,.
\end{align*}
Combining both conditions we see that
\begin{align}\label{eq:t1784yp1t89ytB}
    |\widetilde c_{k+1}-\widetilde c_k|\leq  C 2^{- \alpha k}\quad\mbox{and}\quad |(\widetilde b_{k+1}-\widetilde b_k)|\leq  C 2^{- \alpha k}\log 2^k\,.
\end{align}
Finally, testing \eqref{eq:781236r098froiuavgB} with some $x'\in B_{2^{k}}\setminus \overline{B_{2^{k-1}}}$ satisfying $(\widetilde a_{k+1}-\widetilde a_{k})\cdot x'=|\widetilde a_{k+1}-\widetilde a_{k}||x'|$ and using \eqref{eq:t1784yp1t89ytB}, we deduce the bound for $|\widetilde a_{k+1}-\widetilde a_{k}|$ as well, concluding the proof of \eqref{eq:841tguqB}.

\noindent {\bf Step 2.} Summing the coefficients.

By \eqref{eq:841tguqB}, given $k_0\in \N$, 
$$\sum_{k\geq k_0}|\widetilde c_{k+1}-\widetilde c_k|\leq C_\alpha 2^{-\alpha k_0}\,.$$
Let $\bar \alpha=\max\{\alpha/2,1-2(1-\alpha)\}$; the point is simply that $\bar \alpha<\alpha$ but still $\bar\alpha\to 1$ as $\alpha\to 1$. We can bound then
$$|\widetilde a_{k+1}-\widetilde a_k|\leq C_\alpha 2^{(-\bar\alpha-1)k_0}\quad\mbox{and}\quad|b_{k+1}-b_k|\leq C 2^{- \alpha k}\log 2^k\leq C2^{-\bar\alpha k}\,,$$
which gives
$$\sum_{k\geq k_0}|\widetilde a_{k+1}-\widetilde a_k|\leq C_\alpha 2^{(-\bar \alpha-1) k_0}\quad\mbox{and}\quad \sum_{k\geq k_0} |\widetilde b_{k+1}-\widetilde b_k|\leq C_\alpha 2^{-\bar \alpha k_0}\,.$$

Now, the above show in particular that the sequences $\widetilde a_k,\widetilde b_k,\widetilde c_k$ are Cauchy. This shows the existence of some $\widetilde a_\infty\in\R^2$ and $\widetilde b_\infty, \widetilde c_\infty\in\R$ with
\begin{equation}\label{eq:q7toagvgB}
    |\widetilde a_\infty-\widetilde a_k|\leq C_\alpha 2^{(-\alpha-1)k}\,,\quad\,|\widetilde b_\infty-\widetilde b_k|\leq C_\alpha 2^{-\bar\alpha k}\,,\quad |\widetilde c_\infty-\widetilde c_k|\leq C_\alpha 2^{-\alpha k}\,.
\end{equation}
Substituting into \eqref{eq:37y4tgweeubB}, we deduce that
\begin{equation*}
    |v(x')-\widetilde a_\infty\cdot x'-\widetilde b_\infty-\widetilde c_\infty\log |x'||\leq  C 2^{-\bar\alpha k}\quad\mbox{in}\quad B_{2^{k+1}}\setminus \overline{B_{2^{k-1}}}\,.
\end{equation*}
This means precisely that $v=\widetilde a_\infty\cdot x'+\widetilde b_\infty+\widetilde c_\infty\log |x'|+O(|x'|^{-\bar\alpha})$. Since we are assuming that $|v(x')|\leq C|x'|^{1/2}$, we immediately see that actually $\widetilde a_\infty\equiv 0$. Finally, since $\bar\alpha\to 1$ as $\alpha\to 1$, up to changing the value of $\alpha$ we conclude the proof.
\end{proof}

{\bf Proof of \Cref{prop:mainiterationB}.}

We next show \Cref{prop:mainiterationB}. First, the analog of \Cref{lincomp} is:
\begin{lemma}[Harmonic approximation]\label{lincomp2D} Let $\lambda>0$ and $\alpha \in (0,1)$. There exists $\delta_0>0$, depending on $\lambda, \alpha, \beta$, such that the following holds. Let $v:B_{\frac{1}{\delta_0}}\setminus \overline{B_{\delta_0}}\subset \R^2 \to \R$ satisfy:
\begin{itemize}
    \item $v$ is harmonic.
    
    \item $\|v\|_{L^\infty(B_{2\rho}\setminus \overline{B_\rho})}\leq \rho^{1+\alpha}$, for $1\leq\rho\leq\frac{1}{\delta_0}$.
    
    \item $\|v\|_{L^\infty(B_{2\rho}\setminus \overline{B_\rho})}\leq \rho^{-\alpha}$, for $\delta_0\leq\rho\leq 1$.
\end{itemize}
Then there exist $a\in\R^{2}$ and $b,c\in\R$ such that $\|v-a\cdot x-b-c\log |x|\|_{L^\infty(B_{2}\setminus \overline{B_{1/2}})
}\leq \lambda$.
\end{lemma}
\begin{proof}
Following the same argument as in the proof of \Cref{lincomp}, we reduce the lemma to classifying a harmonic function $v_\infty:\R^2\setminus\{0\}\to\R$ satisfying the same growth conditions as $v$. The difference now is in the application of the Liouville-type classification for harmonic functions in $\R^2\setminus \{0\}$ which satisfy the present growth assumptions---which follows e.g. by separation of variables---since it gives that $v_\infty=a\cdot x+b+c\log |x|$ for some suitable $a,b,c$. The difference with respect to \Cref{lincomp} is that constant functions and $\log |x|$ are harmonic functions in $\R^2\setminus\{0\}$ which grow slower than $|x|^{-\alpha}$ as $|x|\to 0$, making them admissible now in the expansion of $v_\infty$.
\end{proof}
{\bf Preliminaries: Rescaling}

For convenience, we start by considering the new function $\widetilde v:=\frac{1}{2^{k_1}}v(2^{k_1} x)$.
\begin{remark}\label{rmk:89ttqovbaB}
    The original hypotheses---in scales \textbf{from $-k_0$ to $\infty$} now---become, for some suitable new coefficients $\widetilde a_k, \widetilde b_k, \widetilde c_k$:
    \begin{itemize}
    \item Excess decay between $-k_0$ and $0$:
    \begin{equation}\label{eq:52ei7t6siduB}
        |\widetilde v (x')-\widetilde a_k\cdot x'-\widetilde b_k-\widetilde c_k\log|x'||\leq \delta^{k}\ep 2^{k} \mbox{ in } B_{2^{k+1}}\setminus \overline{B_{2^{k-1}}},\,\quad -k_0\leq k \leq 0\,.
    \end{equation}
    \item Preservation of best flatness:
    \begin{equation}\label{eq:52ei7t6sidu2B}
        |\widetilde v (x')-\widetilde a_k\cdot x'-\widetilde b_k-\widetilde c_k\log|x'||\leq \ep 2^{k} \mbox{ in } B_{2^{k+1}}\setminus \overline{B_{2^{k-1}}},\,\quad 1\leq k \,.
    \end{equation}
\end{itemize}
\end{remark}

\begin{lemma}\label{lem:grphregnHypB}
    Let $f:B_{2^{k_0}}\setminus \overline{B_{2^{-k_0}}}\to\R$ be defined by
    \begin{equation}\label{eq:894tibhabHypB}
        f(x'):=\widetilde v (x')-\widetilde a_0\cdot x'-\widetilde b_0-\widetilde c_0\log|x'|\,,
    \end{equation}
    which is also harmonic.

    There is $C_2$ depending on $k_0$ such that
\begin{equation}\label{eq:1278tq3gq}
    f(x')\leq C_2\ep|x'|^{-\frac{1+\alpha}{2}}\quad\mbox{for}\quad 2^{-k_0}\leq|x'|\leq 1
\end{equation}
and
\begin{equation}\label{eq:1278tq3gq2}
    f(x')\leq C_2\ep|x'|\log^2(|x'|+1)\quad\mbox{for}\quad 1\leq|x'|\leq 2^{k_0}\,.
\end{equation}
\end{lemma}
\begin{proof}
Let $\alpha_*=\frac{1+\alpha}{2}>\alpha$. Then, arguing exactly as in the proof where we showed that \Cref{prop:logexp3D} holds (assuming that \Cref{prop:dec2B} does), we obtain:
\begin{equation*}
    |\widetilde a_0-\widetilde a_k|\leq C_\alpha\ep 2^{(-\alpha_*-1)k}\,,\quad\,|\widetilde b_0-\widetilde b_k|\leq C_\alpha \ep 2^{-\alpha_* k}\,,\quad |\widetilde c_0-\widetilde c_k|\leq C_\alpha \ep 2^{-\alpha k}\quad\mbox{for every}\quad -k_0\leq k\leq 0\,,
\end{equation*}
and
\begin{equation*}
    |\widetilde a_0-\widetilde a_k|\leq C_\alpha\ep \log^2 2^k\,,\quad\,|\widetilde b_0-\widetilde b_k|\leq C_\alpha \ep 2^{k}\log 2^k\,,\quad |\widetilde c_0-\widetilde c_k|\leq C_\alpha \ep 2^k\quad\mbox{for every}\quad 1\leq k\,.
\end{equation*}
Substituting into \eqref{eq:52ei7t6siduB}--\eqref{eq:52ei7t6sidu2B}, we deduce that
\begin{equation*}
    |v(x')-\widetilde a_0\cdot x'-\widetilde b_0-\widetilde c_0\log |x'||\leq  C\ep_0 2^{-\alpha_* k}\quad\mbox{in}\quad B_{2^{k+1}}\setminus \overline{B_{2^{k-1}}}
\end{equation*}
and
\begin{equation*}
    |v(x')-\widetilde a_0\cdot x'-\widetilde b_0-\widetilde c_0\log |x'||\leq  C\ep_0 2^{ k}\log^22^k\quad\mbox{in}\quad B_{2^{k+1}}\setminus \overline{B_{2^{k-1}}}\,.
\end{equation*}
Recalling the definition of $f$, we immediately conclude \eqref{eq:1278tq3gq}--\eqref{eq:1278tq3gq2}.
\end{proof}
We can finally give:
\begin{proof}[Proof of \Cref{prop:mainiterationB}]
As above, we consider $\widetilde v(x):=\frac{1}{2^{k_1}}v(2^{k_1} x)$ in place of $v$, which satisfies the hypotheses in \Cref{rmk:89ttqovbaB}. Let $f:B_{2^{k_0}}'\setminus \overline{B_{2^{-k_0}}'}\to\R$ be defined by
$$f(x'):=\widetilde v (x')-\widetilde a_0\cdot x'-\widetilde b_0-\widetilde c_0\log|x'|\,.$$
By \Cref{lem:grphregnHypB},
\begin{equation}\label{eq:89ygqiubrgipabB}
    f(x')\leq C_2\ep|x'|^{-\frac{1+\alpha}{2}}\quad\mbox{for}\quad 2^{-k_0}\leq|x'|\leq 1
\end{equation}
and
\begin{equation}\label{eq:89ygqiubrgipab2B}
    f(x')\leq C_2|x'|\log^2(|x'|+1)\leq C_2|x'|^{1+\frac{1+\alpha}{2}}\quad\mbox{for}\quad 1\leq|x'|\leq 2^{k_0}\,.
\end{equation}

Define the vertical scaling $F(x')=\frac{f(x')}{C_2\ep}$, which is also harmonic.
We want to apply \Cref{lincomp2D} to $F$. Observe that \eqref{eq:89ygqiubrgipabB}--\eqref{eq:89ygqiubrgipab2B} transform into
\begin{equation*}
    F(x')\leq |x'|^{-\frac{1+\alpha}{2}}\quad\mbox{for}\quad 2^{-k_0}\leq|x'|\leq 1
\end{equation*}
and
\begin{equation*}
    F(x')\leq |x'|^{1+\frac{1+\alpha}{2}}\quad\mbox{for}\quad 1\leq|x'|\leq 2^{k_0}\,.
\end{equation*}
Let $\lambda>0$ to be chosen. Let $\delta_0=\delta_0(\lambda,\alpha,\beta,n)\in(0,\lambda]$ be given by \Cref{lincomp2D}.\\

Set $k_0:=\lceil\log_2 \delta_0\rceil$; in particular, $2^{-k_0}\leq \delta_0\leq \frac{1}{\delta_0}\leq 2^{k_0}$.
Note that the hypotheses of \Cref{lincomp2D} are satisfied (e.g. with $\frac{1+\alpha}{2}$ in place of $\alpha$). Hence, we get appropriate $a,b,c$ such that $\|F-a\cdot x'-b-c\log|x'|\|_{L^\infty(B_{\frac{1}{\lambda}}\setminus \overline{B_{\lambda}})
}\leq \lambda$. Equivalently:
\begin{equation*}
    \|f-\bar a\cdot x'-\bar b -\bar c\log|x'|\|_{L^\infty(B_{\frac{1}{\lambda}}\setminus \overline{B_{\lambda}})
}\leq C_2\ep\lambda\quad\mbox{for some}\quad \bar a\in\R^{2}\mbox{ and }\bar b,\bar c\in\R.
\end{equation*}

Recall that we set $k_1'=k_1+1$ and $\ep'=\delta \ep$. Choosing $\lambda=\min\{2\delta C_2^{-1},4\}$ in the previous step, and recalling the definitions of $f$ and $\widetilde v(x)$, we find that
\begin{equation*}
    \|v- a_{k_1'}'\cdot x'- b_{k_1'}' - c_{k_1'}'\log|x'|\|_{L^\infty(B_{4}\setminus \overline{B_{\frac{1}{4}}})
}\leq (\ep 2\delta)2^{k_1}=\ep'2^{k_1'} \quad\mbox{for some new}\quad a_{k_1'}'\in\R^{2}\mbox{ and }b_{k_1'}', c_{k_1'}'\in\R.
\end{equation*}
Now, for $k\leq k_1'-1$ instead, we choose to just keep $(a_k',b_k',c_k'):=(a_k,b_k,c_k)$. This shows the excess decay condition for $k_1'-k_0\leq k \leq k_1'$, and we then conclude the preservation of best flatness part for $k\geq k_1'+1$ just by applying the same reasoning but with $k_1'':=k$ in place of $k_1$, exactly as in the last step of the proof of \Cref{prop:mainiteration}.
\end{proof}


\addcontentsline{toc}{section}{\large References}
\printbibliography

\end{document}